\documentclass{amsart}

\usepackage{amsthm,graphicx,amsmath, verbatim, amssymb, color,mathtools,hyperref,lineno}
\usepackage[margin=1.5in]{geometry}
\modulolinenumbers[5]

\newtheorem{theorem}{Theorem}[section]
\newtheorem{lemma}[theorem]{Lemma}

\newtheorem{definition}[theorem]{Definition}
\newtheorem{corollary}[theorem]{Corollary}
\newtheorem{conjecture}[theorem]{Conjecture}
\newtheorem{proposition}[theorem]{Proposition}
\newtheorem{example}[theorem]{Example}

\theoremstyle{remark}
\newtheorem{remark}[theorem]{Remark}

\newtheoremstyle{TheoremNum}
        { 1em}{1 em}              %%% space between body and thm
        {\itshape}                      %%% Thm body font
        {}                              %%% Indent amount (empty = no indent)
        {\bfseries}                     %%% Thm head font
        {.}                             %%% Punctuation after thm head
        { 0.5 em}                             %%% Space after thm head
        {\thmname{#1}\thmnote{ \bfseries #3}}%%% Thm head spec
    \theoremstyle{TheoremNum}
    \newtheorem{thmn}{Theorem}

 \newtheoremstyle{TheoremNum}
        { 1em}{1 em}              %%% space between body and thm
        {\itshape}                      %%% Thm body font
        {}                              %%% Indent amount (empty = no indent)
        {\bfseries}                     %%% Thm head font
        {.}                             %%% Punctuation after thm head
        { 0.5 em}                             %%% Space after thm head
        {\thmname{#1}\thmnote{ \bfseries #3}}%%% Thm head spec
    \theoremstyle{TheoremNum}
    \newtheorem{coroln}{Corollary}

     \newtheoremstyle{TheoremNum}
        { 1em}{1 em}              %%% space between body and thm
        {\itshape}                      %%% Thm body font
        {}                              %%% Indent amount (empty = no indent)
        {\bfseries}                     %%% Thm head font
        {.}                             %%% Punctuation after thm head
        { 0.5 em}                             %%% Space after thm head
        {\thmname{#1}\thmnote{ \bfseries #3}}%%% Thm head spec
    \theoremstyle{TheoremNum}

\usepackage[latin1]{inputenc}
\usepackage{amsmath,tikz-cd}
\usepackage{setspace}
\usepackage{amsfonts, mathrsfs}
\usepackage{amssymb}

\newcommand{\T}{\mathscr{T}}

\newcommand{\C}{\mathbb{C}}

\usepackage{pdflscape}

\numberwithin{equation}{section}

\begin{document}

\bibliographystyle{amsalpha}

\title{On a conjecture of Voisin on the gonality of very general abelian varieties}

\thanks{The author acknowledges the support of the Natural Sciences and Engineering Research Council of Canada (NSERC). Cette recherche est partiellement financ\'ee par le Conseil de recherches en sciences naturelles et en g\'enie du Canada (CRSNG)}
\author{Olivier Martin}
\email{olivier@uchicago.edu}

\address{Department of Mathematics, University of Chicago, Chicago, IL, USA}


\maketitle

\begin{abstract}
We study orbits for rational equivalence of zero-cycles on very general abelian varieties by adapting a method of Voisin to powers of abelian varieties. We deduce that, for $k$ at least $3$, a very general abelian variety of dimension at least $2k-2$ has covering gonality greater than $k$. This settles a conjecture of Voisin. We also discuss how upper bounds for the dimension of orbits for rational equivalence can be used to provide new lower bounds on other measures of irrationality. In particular, we obtain a strengthening of the Alzati-Pirola bound on the degree of irrationality of abelian varieties.
\end{abstract}

\section{Introduction}
\vspace{0.5em}
In his seminal 1969 paper \cite{M}, Mumford shows that the Chow group of zero-cycles of a smooth complex projective surface with $p_g\neq 0$ is very large. Building on the work of Mumford, in \cite{R1} and \cite{R2} Ro\u{\i}tman studies the map
\begin{align*}\text{Sym}^k (X)&\to CH_0(X)\\
x_1+\ldots+x_k&\mapsto \{x_1\}+\ldots+\{x_k\},
\end{align*}
for a smooth complex projective variety $X$. He shows that fibers of this map are countable unions of Zariski closed subsets. We call orbits of degree $k$ for rational equivalence the fibers of both the map $X^k\to CH_0(X)$ and the map $\text{Sym}^k X\to CH_0(X)$. The orbit of $\{x_1\}+\ldots+\{x_k\}\in CH_0(X)$ is the fiber of $\text{Sym}^k X\to CH_0(X)$ or $X^k\to CH_0(X)$ over this point and is denoted by $|\{x_1\}+\ldots+\{x_k\}|$. Ro\u{\i}tman defines birational invariants $d(X)$ and $j(X)\in \mathbb{Z}_{\geq 0}$ such that for $k\gg 0$ the minimal dimension of orbits of degree $k$ for rational equivalence is $k(\dim X-d(X))-j(X)$. He generalizes Mumford's theorem by showing that
$$H^0(X,\Omega^q)\neq 0\implies d(X)\geq q.$$
In particular, if $X$ has a holomorphic top form, a very general point $x_1+\ldots+x_k\in \text{Sym}^kX$ is contained in a zero-dimensional orbit.\\

Abelian varieties are among the simplest examples of varieties admitting a holomorphic top form. In this article we will focus our attention on this example and take a point of view different from the one mentioned above. Instead of fixing an abelian variety $A$ and considering the minimal dimension of orbits of degree $k$ for rational equivalence, we will be interested in the maximal dimension of such orbits for a very general abelian variety $A$ of dimension $g$ with a given polarization type.\\

This perspective has already been studied by Pirola, Alzati-Pirola, and Voisin (see respectively \cite{P},\cite{AP},\cite{V}) with a view towards the minimal gonality of curves on very general abelian varieties. The story begins with \cite{P} in which Pirola shows that curves of geometric genus less than $\dim A-1$ are rigid in the Kummer variety of a very general abelian variety $A$. This allows him to show: 
\begin{theorem}[Pirola]\label{P}
A very general abelian variety of dimension at least $3$ does not have positive dimensional orbits of degree $2$. In particular, it does not admit a nonconstant morphism from a hyperelliptic curve.
\end{theorem}

There are several ways in which one might hope to generalize this result. For instance, one can ask for the covering gonality of a very general abelian variety of dimension $g$. Given a variety $X$, the covering gonality $\text{cov.gon}(X)$ is the the smallest positive integer $g$ such that a generic point of $X$ is contained in a $g$-gonal curve in $X$. It is a birational invariant that measures how far $X$ is from being uniruled and its significance was observed in \cite{BPEL}. We will call gonality of $X$ the minimal gonality of curves in $X$. For an abelian variety $A$, these two notions coincide since one can use the action of $A$ on itself by translation to produce a covering family of $k$-gonal curves from a single $k$-gonal curve in $A$.\\

Now observe that an irreducible curve $C\subset X$ with gonality $k$ provides us with a positive dimensional orbit in $\text{Sym}^k C$, and thus with a positive dimensional orbit in $\text{Sym}^k X$. Hence, one can give a lower bound on the gonality of $X$ by giving a lower bound on the degree of positive dimensional orbits. This suggests to consider the function
$$\mathscr{G}(k):=\text{inf}\begin{cases}g\in \mathbb{Z}_{>0}\bigg|\begin{rcases} \text{a very general abelian variety of dimension } g \text{ does }\\ \text{ not have a positive dimensional orbit of degree } k\end{rcases},\end{cases} $$
and to attempt to bound it from above. Indeed, a very general abelian variety of dimension at least $\mathscr{G}(k)$ must have gonality at least $k+1$. Note that, given a positive dimensional orbit $|\{a_1\}+\ldots+\{a_k\}|$, the orbit $|\{a_1\}+\ldots+\{a_k\}+\{0_A\}|$ is also positive dimensional. It follows that $\mathscr{G}(k)$ is non-decreasing.\\

In this direction, a few years after the publication of \cite{P}, Alzati and Pirola improved on Pirola's results in \cite{AP}, showing that a very general abelian variety $A$ of dimension at least $4$ does not have positive dimensional orbits of degree $3$, i.e.\;that $\mathscr{G}(3)\leq 4$. This suggest that, for any $k\in \mathbb{Z}_{>0}$, a very general abelian variety of sufficiently large dimension does not admit a positive dimensional orbit of degree $k$, i.e.\;that $\mathscr{G}(k)$ is finite for any $k\in \mathbb{Z}_{>0}$. This problem was posed in \cite{BPEL} and answered positively by Voisin in \cite{V}.

\begin{theorem}[Voisin, Thm. 1.1 in \cite{V}]\label{Vmainthm}
A very general abelian variety of dimension at least $2^{k-2}(2k-1)+(2^{k-2}-1)(k-2)$ does not have a positive dimensional orbit of degree $k$, i.e. $\mathscr{G}(k)\leq 2^{k-2}(2k-1)+(2^{k-2}-1)(k-2).$
\end{theorem}
Voisin provides some evidence suggesting that this bound can be improved significantly. Her main conjecture from \cite{V} is the following linear bound on the gonality of very general abelian varieties:
\begin{conjecture}[Voisin, Conj. 1.2 in \cite{V}]A very general abelian variety of dimension at least $2k-1$ has gonality at least $k+1$.\label{Vweakconj}
\end{conjecture}

The main result of this paper is the proof of this conjecture. It is obtained by generalizing Voisin's method to powers of abelian varieties. This allows us to rule out the existence of positive dimensional orbits of degree $k$ for very general abelian varieties of sufficiently large dimension.
\begin{thmn}[\ref{2k-2}]
For $k\geq 3$, a very general abelian variety of dimension at least $2k-2$ has no positive dimensional orbits of degree $k$, i.e.
$\mathscr{G}(k)\leq 2k-2$.
\end{thmn}

This gives the following lower bound on the gonality of very general abelian varieties:
\begin{coroln}[\ref{gonalitybound}]
For $k\geq 3$, a very general abelian variety of dimension at least $2k-2$ has gonality at least $k+1$. In particular, Conjecture \href{Vweakconj}{\ref{Vweakconj}} holds.
\end{coroln}

We can also generalize Theorem \href{P}{\ref{P}}\; in a different direction. Observe that a nonconstant morphism from a hyperelliptic curve $C$ to an abelian surface $A$ gives rise to a positive dimensional orbit of the form $|2\{0_A\}|$. Indeed, translating $C$ if needed, we can assume that a Weierstrass point of $C$ maps to $0_A$. This suggests to consider the maximal $g$ for which a very general abelian variety $A$ of dimension $g$ contains an irreducible curve $C\subset A$ whose normalization $\pi: \widetilde{C}\to C$ admits a degree $k$ morphism to $\mathbb{P}^1$ with a point of ramification index at least $l$. We say that $p\in \widetilde{C}$ has ramification index at least $l$ if the sum of the ramification indices of $\widetilde{C}\to \mathbb{P}^1$ at all points in $\pi^{-1}(\pi(p))$ is at least $l$. This maximal $g$ is less than
$$\mathscr{G}_l(k):=\text{inf}\begin{cases}g\in \mathbb{Z}_{>0}\bigg| \begin{rcases}\text{a very general abelian variety of dimension } g \text{ does not admit }\\\text{ a positive dimensional orbit of the form } |\sum_{i=1}^{k-l} \{a_i\}+l\{a_0\}|\end{rcases}.\end{cases}$$
Clearly, we have
$$\mathscr{G}_1(k)=\mathscr{G}_0(k)=\mathscr{G}(k),$$
and
$$\mathscr{G}_l(k)\leq \mathscr{G}_{l-1}(k),\qquad\forall l\in \{1,\ldots, k\}.$$
In this direction, Voisin shows:
\begin{theorem}[Voisin, Thm. 1.4 (iii) \& (iv) in \cite{V}] The following inequality holds:
$$\mathscr{G}_l(k)\leq 2^{k-l}(2k-1)+(2^{k-l}-1)(k-2).$$
In particular,
$$\mathscr{G}_k(k)\leq 2k-1.$$
\end{theorem}
We can improve on Voisin's result to show the following:

\begin{thmn}[\ref{k+1}]
A very general abelian variety $A$ of dimension at least $2k+2-l$ does not have a positive dimensional orbit of the form $|\sum_{i=1}^{k-l}\{a_i\}+l\{0_A\}|$, i.e. $\mathscr{G}_{l}(k)\leq 2k+2-l.$ \\
Moreover, if $A$ is a very general abelian variety of dimension at least $k+1$, the orbit $|k\{0_A\}|$ is countable, i.e. $\mathscr{G}_{k}(k)\leq k+1.$
\end{thmn}

Taking a slightly different perpective, one can consider the maximal dimension of orbits of degree $k$ for abelian varieties. One of the main contributions of \cite{V} is the following extension of results of Alzati-Pirola (the cases $k=2,3$, see \cite{AP}):
\begin{theorem}[Voisin, Thm. 1.4 (i) of \cite{V}]\label{k-1}
Orbits of degree $k$ on an abelian variety have dimension at most $k-1$.
\end{theorem}
As observed by Voisin, by considering abelian varieties containing an elliptic curve, one sees that this bound cannot be improved for all abelian varieties. Moreover, it cannot be improved for very general abelian surfaces as shown in \cite{HSL} (see Example \hyperref[HSL]{\ref{HSL}} below). Nevertheless, Theorem \href{Vmainthm}{\ref{Vmainthm}} shows that it can be improved for very general abelian varieties of large dimension. In this direction, we have:

\begin{thmn}[\ref{no(k-1)}]
Orbits of degree $k$ on a very general abelian variety of dimension at least $k-1$ have dimension at most $k-2$.
\end{thmn}

Recall that the degree of irrationality $\text{irr}(X)$ of a variety $X$ is the minimal degree of a dominant rational map $X\dashrightarrow \mathbb{P}^{\dim X}$. The previous theorem provides the following improvement of the Alzati-Pirola bound $\text{irr}(A)\geq \dim A+1$ on the degree of irrationality of abelian varieties (see \cite{AP2}):

\begin{coroln}[\ref{sommese}]
If $A$ is a very general abelian variety of dimension at least $3$, then 
$$\textup{irr}(A)\geq \dim A+2.$$
\end{coroln}

In the last part of this introduction we will sketch Voisin's proof of Theorem \href{Vmainthm}{\ref{Vmainthm}} and give an outline of the paper. Voisin considers what she calls \textit{naturally defined subsets of abelian varieties} (see \cite{V} Definition 2.1). Consider $\mathcal{A}/S$, a locally complete family of abelian varieties of dimension $g$ with a fixed polarization type $\theta$. There are subvarieties $S_\lambda\subset S$ along which $\mathcal{A}_s\sim \mathcal{B}_s^{\lambda}\times \mathcal{E}_s^{\lambda}$, where $\mathcal{B}^{\lambda}/S_\lambda$ is a family of abelian varieties of dimension $(g-1)$, and $\mathcal{E}^\lambda/S_\lambda$ is a family of elliptic curves. The index $\lambda$ keeps track of the isogeny. Let $S_\lambda(B):=\{s\in S_\lambda: \mathcal{B}^\lambda_s\cong B\}$, and $p_\lambda: \mathcal{A}_s\to \mathcal{B}^\lambda_s\times \mathcal{E}_{s}^{\lambda} \to \mathcal{B}^\lambda_s$ be the composition of the projection to $\mathcal{B}_s^\lambda$ with the isogeny.\\

Voisin shows that, given a naturally defined subset $\Sigma_{\mathcal{A}}\subsetneq \mathcal{A}$, there is a locus $S_\lambda\subset S$ such that the image of the restriction of $p_\lambda: \mathcal{A}_s\to \mathcal{B}^\lambda_s\cong B$ to $\Sigma_{\mathcal{A}_{s}}$ varies with $s\in S_\lambda(B)$, for a generic $B$ in the family $\mathcal{B}^{\lambda}$. This shows that if $\Sigma_{B}$ is at most $d$-dimensional for a generic $(g-1)$-dimensional abelian variety $B$ regardless of polarization type and $0<d<g-1$, then $\Sigma_A$ is at most $(d-1)$-dimensional for a very general abelian variety $A$ of dimension $g$ regardless of polarization type.\\

She proceeds to show that the set
 $$\bigcup_{i=1}^k \text{pr}_i(|k\{0_A\}|)=\left\{a_1\in A: \exists a_2,\ldots, a_k\in A\text{ such that }\sum_{i=1}^k\{a_i\}= k\{0_A\}\in CH_0(A)\right\}$$
 is contained in a naturally defined subset $A_k\subset A$ and that $\dim A_k\leq k-1$. In particular, $A_k\neq A$ if $\dim A\geq k$. It follows from the discussion above that 
 $$\dim\bigcup_{i=1}^k \text{pr}_i(|k\{0_A\}|)=0$$ if $\dim A\geq 2k-1$ and $A$ is very general. An induction argument finishes the proof by passing from the orbit $|k\{0_A\}|$ to general orbits of degree $k$. Voisin's results are very similar in spirit to those of \cite{AP} but a key difference is that the latter is concerned with subvarieties of $A^k$ and not of $A$. We will see that this has important technical consequences. While Lemma 2.5 in \cite{V} shows that the restriction of the projection $p^\lambda_{s}: \mathcal{A}^\lambda_s\to \mathcal{B}^\lambda_s$ to $\Sigma_{\mathcal{A}^\lambda_s}$ is generically finite on its image, this becomes a serious sticking point in \cite{AP} (see Lemmas 6.8 to 6.10 of \cite{AP}). One of the main innovations of this article is a way to obtain an analogous generic finiteness result in the setting of subvarieties of $A^k$. This allows us to carry out an induction argument to prove Conjecture \ref{Vweakconj}.\\
 
Consider $\mathcal{A}/S$, a locally complete family of abelian $g$-folds, along with a family $\mathcal{Z}\subset \mathcal{A}_S^k$ of irreducible $d$-dimensional subvarieties. Assume further that for every $s\in S$, the subvariety $\mathcal{Z}_s$ lies in the union of positive-dimensional fibers of the map $\mathcal{A}_s^k\to CH_0(\mathcal{A}_s)$ (i.e. $\mathcal{Z}$ is foliated by positive-dimensional suborbits, see Definitions \ref{suborbit} and \ref{famsuborbits}). As in \cite{V}, we can specialize to loci $S_\lambda$ along which $\mathcal{A}_s$ is isogenous to $\mathcal{B}^\lambda_s\times \mathcal{E}^\lambda_s$, where $\mathcal{B}^\lambda/S_\lambda$ is a locally complete family of abelian $(g-1)$-fold and $\mathcal{E}^\lambda_s$ is a locally complete family of elliptic curves. Let $S_\lambda(B)\subset S_\lambda$ be the locus where $\mathcal{B}^\lambda_s$ is isomorphic to a fixed $B$ and the map
$$p_\lambda: \mathcal{A}_s^k\to (\mathcal{B}_s^\lambda\times \mathcal{E}_s^\lambda)^k\to (\mathcal{B}_s^\lambda)^k$$
be the $k$-fold cartesian product of the composition of the projection to $\mathcal{B}_s^\lambda$ with the isogeny. Our main technical result is Proposition \ref{prop1}, which is analogous to Proposition 2.4 of \cite{V}. It states that there is a $\lambda$ such that the variety $p_\lambda(\mathcal{Z}_s)\subset B^k$ varies with $s\in S_\lambda(B)$, for generic $B$ in the family $\mathcal{B}^\lambda$. This requires some technical non-degeneracy conditions on the family $\mathcal{Z}$ which are not obviously satisfied in our situation.\\

Setting aside these technicalities, our strategy to prove Conjecture \ref{Vweakconj} is the following: Consider $\mathcal{A}_1/S_1$, a locally complete family of abelian $(2k-1)$-folds, and suppose that the set 
$$\{s\in S_1: \mathcal{A}_{1,s} \text{ contains a }k\text{-gonal curve}\}\subset S_1$$
is not contained in a countable union of Zariski-closed subsets of $S_1$. Up to passing to a generically finite cover over $S_1$, we get a subvariety $\mathcal{Z}_1\subset(\mathcal{A}_{1})_{S_1}^k$ whose fibers over $S_1$ are irreducible curves, and such that $\mathcal{Z}_{1,s}$ is contained in a fiber of $\mathcal{A}_{1,s}^k\to CH_0(\mathcal{A}_{1,s})$ for all $s\in S_1$. Moreover, using translation we can arrange for $\mathcal{Z}_1$ to be in the kernel of the summation map $(\mathcal{A}_1)_{S_1}^{k}\to \mathcal{A}_1$. In light of the projection argument above, we can find a locally complete family of abelian $(2k-2)$-folds $\mathcal{A}_2/S_{2}$, and a subvariety $\mathcal{Z}_2\subset (\mathcal{A}_{2})_{S_2}^k$ whose fibers over $S_{2}$ are irreducible surfaces, and such that $\mathcal{Z}_{2,s}$ is contained in the union of positive-dimensional fibers of $(\mathcal{A}_{2,s})^k\to CH_0(\mathcal{A}_{2,s})$ for every $s\in S_{2}$.\\

We can hope to proceed by induction to obtain for each $1\leq i \leq 2k-2$ a locally complete family of abelian $(2k-i)$-folds $\mathcal{A}_i/S_i$, along with subvarieties $\mathcal{Z}_i\subset (\mathcal{A}_i)_{S_i}^k$ satisfying:
\begin{itemize}
\item The fibers of $\mathcal{Z}_i/S_i$ are irreducible of dimension $i$,
\item $\mathcal{Z}_{i,s}$ is contained in the union of the positive-dimensional fibers of $\mathcal{A}_{i,s}^k\to CH_0(\mathcal{A}_{i,s})$ for every $s\in S_i$,
\item The subvariety $\mathcal{Z}_i$ is contained in the kernel of the summation map $\mathcal{A}_i^k\to \mathcal{A}_i$.
\end{itemize}

It would then follow that $\mathcal{Z}_{2k-2}\subset (\mathcal{A}_{2k-2})_{S_{2k-2}}^k$ has fibers of dimension $2k-2$ over $S_{2k-2}$. Since the kernel of the summation map $(\mathcal{A}_{2k-2})_{S_{2k-2}}^k\to \mathcal{A}_{2k-2}$ has dimension $2k-2$, the subvariety $\mathcal{Z}_{2k-2}$ must coincide with this kernel. Using translation, we see that every point of $(\mathcal{A}_{2k-2})_{S_{2k-2}}^k$ is contained in a positive-dimensional orbit. This contradicts results of Ro\u\i tman on rational equivalence on surfaces with $p_g\neq 0$ (see \cite{R2}) and proves Conjecture \ref{Vweakconj}.\\

In Section \ref{prelim}, we will fix some notation and state some important facts about orbits for rational equivalence on surfaces with $p_g\neq 0$ and abelian varieties. The technical heart of the paper consists of Sections \ref{specialization} and \ref{salvaging}. In Section \ref{specialization}, we prove that if the family of irreducible subvarieties $\mathcal{Z}\subset\mathcal{A}_S^k$ satisfies certain non-degeneracy conditions, then there is a choice of $S_\lambda$ such that $p_\lambda(\mathcal{Z}_s)\subset B^k$ varies with $s\in S_\lambda(B)$, for a generic $B$ in the family $\mathcal{B}^\lambda$. Note that for use in Section \ref{salvaging} we will need this result for specializations to loci of abelian varieties containing an abelian $l$-fold for any $1\leq l\leq g/2$. Proposition \ref{prop1} is thus stated in this generality. In Section \ref{salvaging}, we show how to use the inductive nature of the argument sketched above to obtain the non-degeneracy conditions of the families of varieties $\mathcal{Z}_i\subset (\mathcal{A}_i)^k_{S_i}$ appearing at each step of the induction. To achieve this, we consider simultaneously several specializations and projections of the type $p_\lambda$. It allows us to complete the proof of Theorem \ref{2k-2}. Finally, in Section \ref{further} we mention applications of our results to other measures of irrationality for abelian varieties, improving on the Alzati-Pirola bound on the degree of irrationality of abelian varieties. We also discuss an existence result for positive dimensional orbits on abelian varieties and formulate Conjecture \ref{Vstrongconj}, a strengthening of Voisin's conjecture which would follow from another conjecture formulated in \cite{V}.

\vspace{1em}

\section{Preliminaries}\label{prelim}
\vspace{0.5em}

In this section we fix some notation and establish some facts about positive dimensional orbits for rational equivalence on abelian varieties.
\subsection{Notation}
\vspace{1em}
We call variety a locally closed subset of a reduced scheme of finite type over $\mathbb{C}$. In what follows, $X$ is a smooth projective variety, $\mathcal{X}/S$ is a family of such varieties, $A$ is an abelian $g$-fold with polarization type $\theta$, and $\mathcal{A}/S$ is a family of such abelian varieties. In an effort to simplify notation, we will write $\mathcal{X}^k$ instead of $\mathcal{X}_S^k$ to denote the $k$-fold fiber product of $\mathcal{X}$ with itself over $S$. Moreover, $CH_0(X)$ will denote the Chow group of zero-cycles of a smooth projective variety $X$ with $\mathbb{Q}$-coefficients.\\

Given $\mathcal{A}/S$, a family of abelian $g$-folds with polarization type $\theta$, we will abuse notation to denote by $\varphi_{\mathcal{A}}$ both the morphism from $S$ to the moduli stack of abelian $g$-folds with polarization type $\theta$ and the $k$-fold cartesian product of the natural map between $\mathcal{A}$ and the universal family over this stack, for any $k\in \mathbb{Z}_{>0}$. It will be clear from the context what is the meaning of $\varphi_{\mathcal{A}}$.

\begin{definition}
A family $\mathcal{A}/S$ of abelian $g$-folds with polarization type $\theta$ is locally complete if the induced morphism $\varphi_{\mathcal{A}}$ from $S$ to the moduli stack of abelian $g$-folds with polarization type $\theta$ is generically finite. It is almost locally complete if this morphism is dominant. 
\end{definition}

In particular, if $\mathcal{A}/S$ is an almost locally complete family of abelian varieties, then $\varphi_{\mathcal{A}}(\mathcal{A})/\varphi_{\mathcal{A}}(S)$ is a locally complete family of abelian varieties. Note that the following remark will allow us to only ever deal with varieties and avoid stacks entirely.

\begin{remark}\label{convention}
In many of our arguments we will consider a family of varieties $\mathcal{X}\to S$ and a subvariety $\mathcal{Z}\subset \mathcal{X}$ such that $\mathcal{Z}\to S$ is flat with irreducible fibers. We will often need to base change by a generically finite morphism $S'\to S$. To limit the growth of notation, we will denote the base changed family by $\mathcal{X}\to S$ again. Moreover, if $S_\lambda\subset S$ is a Zariski closed subset, the base change of $S_\lambda$ under $S'\to S$ will also be denoted $S_\lambda$. Note that this applies also to the statements of theorems. For instance, if we say a statement holds for a family $\mathcal{X}/S$, we mean that it holds for some family $\mathcal{X}_{S'}/S'$, where $S'\to S$ is generically finite.
\end{remark}

\vspace{1em}
\subsection{Suborbits}
\vspace{1em}
To avoid dealing with countable unions of subvarieties, it is often convenient to consider subvarieties of orbits. This motivates the following definition:
\begin{definition}\label{suborbit}
A suborbit of degree $k$ for $X$ is a Zariski closed subset of $X^k$ contained in a fiber of $X^k\to CH_0(X)$. A suborbit of degree $k$ for $\mathcal{X}$ is a Zariski closed subset $\mathcal{Z}\subset \mathcal{X}^k$ such that $\mathcal{Z}_s$ is a suborbit of $\mathcal{X}_s^k$ for every $s\in S$.

\end{definition}
The notion of suborbit is closely related to but not to be confused with that of constant cycle subvarieties in the sense of Huybrechts (see \cite{H}). Indeed, a suborbit of degree $1$ is exactly the analogue of a constant cycle subvariety as defined by Huybrechts for K3 surfaces. Nonetheless, a suborbit of degree $k$ for $X$ need not be a constant cycle subvariety of $X^k$; in the former case we consider rational equivalence of cycles in $X$, while in the latter rational equivalence of cycles in $X^k$.\\

We will not only be interested in suborbits but in families of suborbits and subvarieties of $X^k$ foliated by suborbits. An $r$-parameter family of effective $d$-dimensional cycles on a variety $X$ is an $r$-dimensional subvariety of a Chow variety $\text{Chow}_d(X)$ which parametrizes effective pure $d$-dimensional cycles with a given cycle class. For simplicity, we will leave the cycle class unspecified in the notation for these Chow varieties. Similarly, an $r$-parameter family of effective $d$-dimensional cycles on $\mathcal{X}/S$ is an irreducible $r$-dimensional subvariety $\mathcal{D}\subset \text{Chow}_d(\mathcal{X}/S)$ with relative dimension $r$ over $S$. The subvariety $\mathcal{D}_s\subset \mathcal{D}$ parametrizes cycles with support lying over a point $s\in S$.

%Recall from the introduction the following:
%\begin{definition}
%Given $\underline{x}=(x_1,\ldots, x_k)\in X^k$ the \emph{orbit} of $\underline{x}$ is the set
%$$|\underline{x}|:=\{(x_1',\ldots, x_k')\in X^k: \sum \{x_i\}=\sum \{x_i'\}\in CH_0(X)\}$$
%\end{definition}

%\begin{definition}
%A subvariety $Z\subset X^k$ is called a \emph{constant cycle subvariety of degre} (CCS) if for any $(x_1,\ldots, x_k),\;(x_1',\ldots, x_k')$ we have 
%$$\sum_{i=1}^k \{x_i\}=\sum_{i=1}^k\{x_i'\}\in CH_0(X)$$ i.e. for any $\underline{x}=(x_1,\ldots, x_k)\in X^k$ we have $Z\subset  |\underline{x}|$. Note that this should not be confused with the requirement that $\{(x_1,\ldots, x_k)\}=\{(x_1',\ldots, x_k')\}\in CH_0(X^k)$ which is also used in the literature.
%\end{definition}

%\begin{definition}
%Given a subvariety $Z\subset X^k$ we say it is \emph{foliated by $d$-dimensional CCSs} if a generic $z\in Z$ is such that $|z|\cap Z$ is $d$-dimensional.\\

%Similarly, we say a subvariety $\mathcal{Z}\subset \mathcal{X}_S^k$ is \emph{foliated by $d$-dimensional CCSs} if for generic $s\in S$ the Zariski closed subset $\mathcal{Z}_s\subset \mathcal{X}_s^k$ is foliated by $d$-dimensional CCSs.
%\end{definition}

\begin{definition}\label{famsuborbits}
$\;$
\begin{enumerate}
\item
A locally closed subset $Z\subset X^k$ is foliated by $d$-dimensional suborbits if
$$\dim |z|\cap Z\geq d,\qquad \forall z\in Z.$$

\item Similarly, a locally closed subset $\mathcal{Z}\subset \mathcal{X}^k$ is foliated by $d$-dimensional suborbits if $\mathcal{Z}_s$ is foliated by $d$-dimensional suborbits for all $s\in S$.

\item A family of $d$-dimensional suborbits of degree $k$ for $X$
 is a family $D\subset \textup{Chow}_d(X^k)$ of $d$-dimensional effective cycles on $X^k$ whose supports are suborbits.

\item Similarly, a family of $d$-dimensional suborbits of degree $k$ for $\mathcal{X}/S$ is a family $\mathcal{D}\subset \textup{Chow}_d(\mathcal{X}^k/S)$ of effective $d$-dimensional cycles on $\mathcal{X}^k$, such that $\mathcal{D}_s\subset\textup{Chow}_d(\mathcal{X}_s^k)$ is a family of $d$-dimensional suborbits of degree $k$ for $\mathcal{X}_s$.
\end{enumerate}
\end{definition}

Note that given $D\subset \text{Chow}_d(X^k)$, an $r$-parameter family of $d$-dimensional suborbits for a variety $X$, and $\mathcal{Y}\to D$, the corresponding family of cycles, the set $\bigcup_{t\in D}\mathcal{Y}_t\subset X^k$ is foliated by $d$-dimensional suborbits. Yet, its dimension might be less than $(r+d)$. Conversely, any subvariety $Z\subset X^k$ foliated by $d$-dimensional suborbits arises in such a fashion from a family of $d$-dimensional suborbits after possibly passing to a Zariski open subset. Indeed, by work of Ro\u{\i}tman (see \cite{R1}), the set
$$\Delta_{\text{rat}}=\left\{((x_1,\ldots, x_k),(x_1',\ldots, x_k'))\in X^k\times X^k: \sum_{i=1}^k x_i=\sum_{i=1}^k x_i'\in CH_0(X)\right\}\subset X^k\times X^k$$
is a countable union of Zariski closed subsets. Consider
$$\pi: \Delta_{\text{rat}}\cap Z\times Z\to Z,$$
the restriction of the projection of $Z\times Z$ onto the first factor to $\Delta_{\text{rat}}\cap Z\times Z$. Given an irreducible component $\Delta'$ of $\Delta_{\text{rat}}\cap Z\times Z$ which dominates $Z$ and has relative dimension $d$ over $Z$, we can consider the map from an open set in $Z$ to an appropriate Chow variety of $X^k$ taking $z\in Z$ to the cycle $[\Delta_z']$. Letting $D$ be the image of this morphism, we get a family of $d$-dimensional suborbits that covers an open in $Z$.\\

\subsection{Subvarieties foliated by suborbits}

For a variety $X$, we denote by $\text{pr}_i: X^k\to X$ the projection to the $i^{\text{th}}$ factor. Given a form $\omega\in H^0(X,\Omega^q)$, we let $\omega_k:=\sum_{i=1}^k \text{pr}_i^*\omega\in H^0(X^k, \Omega^q)$. For the sake of simplicity, we mostly deal with $X^k$ rather than $\text{Sym}^kX$ and we take the liberty to call points of $X^k$ effective zero-cycles of degree $k$. In Corollary 3.5 of \cite{AP}, the authors show the following generalization of Mumford's theorem on rational equivalence of zero-cycles on surfaces with $p_g\neq 0$ (see \cite{M}):

\begin{proposition}[Alzati-Pirola, Corollary 3.5 in \cite{AP}]\label{vanish}
Let $D$ be an an $r$-parameter family of suborbits of degree $k$ for a variety $X$. Let $\mathcal{Z}\to D$ be the corresponding family of cycles and $g: \mathcal{Z}\to X^k$ the natural map. If $\omega\in H^0(X,\Omega^q)$ and $q>r$, then $g^*(\omega_k)=0.$
\end{proposition}

\begin{definition}
Given an abelian variety $A$, the zero-cycle $(a_1,\ldots, a_k)\in A^k$ or the orbit $|\{a_1\}+\ldots+\{a_k\}|$ is called normalized if $a_1+\ldots+a_k=0_A$. We write $A^{k,\,0}$ for the kernel of the summation map, i.e. the set of normalized effective zero-cycles of degree $k$.
\end{definition}

\begin{remark}\label{trick}
If $|\{a_1\}+\ldots+\{a_k\}|$ as above is a normalized $d$-dimensional orbit of degree $k$ and $a\in A$ is any point, then $|\{ka\}+\{a_1-a\}+\ldots+\{a_k-a\}|$ is a normalized orbit of dimension at least $d$. This was observed by Voisin in Example 6.3 of \cite{V}.
\end{remark}

In Proposition 4.2 of \cite{V}, Voisin shows that if $A$ is an abelian variety and $Z\subset A^k$ is such that $\omega_k|_{Z}=0$ for all $\omega\in H^0(A,\Omega^q)$ and all $q\geq 1$, then $\dim Z\leq k-1$. Note that $\omega_k=m^*(\omega)$ for any $\omega\in H^0(A,\Omega^1)$, where $m: A^k\to A$ is the summation map. Hence, Voisin's result along with Proposition \ref{vanish} implies the following:

\begin{corollary}\label{1dimfam}
An abelian variety does not admit a one-parameter family of $(k-1)$-dimensional normalized suborbits of degree $k$.
\end{corollary}

This non-existence result along with our degeneration and projection argument will provide the proof of Theorem \href{no(k-1)}{\ref{no(k-1)}}. For most other applications, we will use non-existence results for large families of suborbits on surfaces with $p_g\neq 0$.

\begin{lemma}\label{vanish2}
Let $X$ be a smooth projective surface and $\omega\in H^0(X,\Omega^2_X)$. If $Z\subset X^k$ is a Zariski closed subset of dimension $m$ foliated by $d$-dimensional suborbits, then $\omega_k^{\lceil(m-d+1)/2\rceil}$ restricts to zero on $Z$.
\end{lemma}
\begin{proof}
The set of points $z\in Z$ such that $z\in Z_{\text{sm}} \cap(|z|\cap Z)_{\text{sm}}$ is clearly Zariski dense. Thus, it suffices to show that $\omega_k^{\lceil(m-d+1)/2\rceil}$ restricts to zero on $T_{Z,z}$ for such a $z$. Suppose that $m-d$ is odd (the even case is treated in a similar way). Any $(m-d+1)$-dimensional subspace of $T_{Z,z}$ meets the tangent space to $T_{|z|\cap Z,z}$ and so admits a basis $v_1,\ldots, v_{m-d+1},$ with $v_1\in T_{|z|\cap Z,z}$. The number $\omega_k^{(m-d+1)/2}(v_1,\ldots, v_{m-d+1})$
consists of a sum of terms of the form 
$$\pm \prod_{i\in I}\omega_k(v_i,v_{\sigma(i)}),$$
where $I\subset \{1,\ldots, m-d+1\}$ is a subset of cardinality $(m-d+1)/2$ and
$$\sigma: I\to \{1,\ldots, m-d+1\}\setminus I$$
is a bijection. But $\omega_k(v_1,v_j)=0$ for any $j$ by Proposition \href{vanish}{\ref{vanish}}. Indeed, the space spanned by $v_1$ and $v_j$ is contained in the tangent space to a $(d+1)$-fold foliated by at least $d$-dimensional suborbits. It follows that $\omega_k^{(m-d+1)/2}(v_1,\ldots, v_{m-d+1})=0$ and so ${\omega_k^{(m-d+1)/2}}$ restricts to zero on $Z$.
\end{proof}

\begin{lemma}\label{surfbound}
Let $X$ be a Calabi-Yau surface and $\omega\in H^0(X,\Omega^2)$ be non-zero. If $\omega_k^l$ restricts to zero on $Z\subset X^k$, then $\dim Z< k+l$.
\end{lemma}
\begin{proof}
Pick $z\in Z_{\text{sm}}$ and let $v_1,\ldots, v_m$ be a basis of $T_{Z,z}$. Since $\omega_k^k$ is a holomorphic volume form on $X^k$, we have $\omega^k_k(v_{1},\ldots, v_{m})\neq 0\in \bigwedge^{2k-m}T_{X^k,z}^*$. If $m\geq k+l$, then $\omega_k^k(v_{1},\ldots, v_{m})$ is a sum of terms of the form
$$\pm\omega_k^l(v_{i_1},\ldots, v_{i_{2l}})\cdot \omega_k^{k-l}(v_{i_{2l+1}},\ldots, v_{i_m}),$$
where $\{i_1,\ldots, i_{2l+m}\}=\{1,\ldots, m\}$.
It follows from the non-vanishing of $\omega^k_k(v_{1},\ldots, v_{m})$ that $\omega_k^l$ does not restrict to zero on $T_{Z,z}$.
\end{proof}

Given an abelian variety $A$, we denote by $A^r_{M}$ (or $A_M$ when $r=1$) the image of $A^r$ under the following embedding:
\begin{align*}
i_M: A^r&\xrightarrow{\hspace*{2.5cm}} A^k\\
(a_1,\ldots, &a_r)\mapsto \Big(\sum_{j=1}^r m_{1j}a_j,\ldots, \sum_{j=1}^r m_{kj}a_j\Big),\end{align*}
where $M=(m_{ij})\in M_{k\times r}(\mathbb{Z})$ has rank $r$. Note that if $A$ is simple, all abelian subvarieties of $A^k$ are of the form $A^r_M$ for some matrix $M$ of rank $r$. 
 Given a $\mathbb{C}$-vector space $V$, we will use the same notation $V_M^r:=i_M(V)\subset V^k$, for  $M\in M_{k\times r}(\C)$.
\begin{lemma}\label{absub}
Let $A$ be an abelian variety and $\omega\neq 0\in H^0(A,\Omega^2)$. Then
$$(i_M^*\omega_k)^{ r}\neq 0  \in H^0(A_M^r,\Omega^{2r}).$$
In particular, if $\dim A=2$, the form $\omega_k$ restricts to a holomorphic symplectic form on $A_M^r$.
%Should check if the corresponding statement is true for any abelian variety of even dimension.
\end{lemma}
\begin{proof}
We only treat the case $\omega=dz\wedge dw$, where $z$ and $w$ are two coordinates on $A$. The general case is easily obtained from this special case. Let $z_{i},w_i$ be the corresponding coordinates on the $i^{\text{th}}$ factor of $A^r$. We have 
$$i_M^*\omega_k=\sum_{i=1}^k\left(\sum_{j=1}^rm_{ij}dz_{j}\wedge \sum_{j'=1}^rm_{ij'}dw_{j'}\right).$$
A computation gives
$$(i_M^*\omega_k)^{r}=r!\det \mathbf{G}\; dz_{1}\wedge dw_1\wedge\ldots \wedge dz_{r}\wedge dw_{r},$$
where $\mathbf{G}=(\langle M_i,M_j\rangle)_{1\leq i,j\leq r}$ is the Gram matrix of the columns $M_i$ of the matrix $M$. Since $M$ has maximal rank its Gram matrix has positive determinant. It follows that $(i_M^*\omega_k)^{ r}\neq 0$.
\end{proof}

Lemmas \href{vanish2}{\ref{vanish2}} and \ref{absub} imply the following corollary which will play a crucial technical role in our argument:

\begin{corollary}\label{notfoliated2}
Given an abelian surface $A$, a subvariety of the form $A^r_M\subset A^k$ cannot be foliated by positive dimensional suborbits.
\end{corollary}

The main result on rational equivalence of zero-cycles that we will use in the proof of Theorem \ref{2k-2} is a version of Lemma \ref{surfbound} for normalized suborbits of abelian surfaces. Using the same argument as in Lemma \href{surfbound}{\ref{surfbound}}, we get the following:

\begin{lemma}\label{surfbound2}
Let $A$ be an abelian surface and $Z\subset A^{k,\,0}$ be such that $\omega_k^l$ restricts to zero on $Z$. Then, 
$$\dim Z< k+l-1.$$
\end{lemma}

\begin{corollary}\label{normabsurfbound}
Let $A$ be an abelian surface and $Z\subset A^{k,\,0}$ be foliated by $d$-dimensional suborbits. Then, 
$$\dim Z\leq 2(k-1)-d.$$\label{absurfbound}
\end{corollary}
\begin{proof}
By Lemma \href{vanish2}{\ref{vanish2}}\; $\omega^{\lceil (\dim Z-d+1)/2\rceil}$ restricts to zero on $Z$. Then by Lemma \href{surfbound2}{\ref{surfbound2}}\; 
$$\dim Z<k+\lceil(\dim Z-d+1)/2\rceil-1.$$
This gives the stated bound for parity reasons.
\end{proof}

\vspace{0.5em}

\section{Specialization and projection}\label{specialization}
\vspace{0.5em}
In this section we generalize Voisin's method from Section 2 of \cite{V} to powers of abelian varieties. The key difference is that our generalization requires technical assumptions which are automatically satisfied in Voisin's setting. Checking that these assumptions are also satisfied in the case at hand occupies the bulk of the next section.\\

\subsection{Setup and conditions $(*)$, $(**)$, and $(***)$}

Just as in \cite{P},\cite{AP}, and \cite{V}, we will consider specializations to loci of non-simple abelian varieties. Given $\mathcal{A}/S$, a locally complete family of abelian varieties of dimension $g$, and a positive integer $l< g$, we denote by $S_{\lambda}\subset S$ the loci along which 
$$\mathcal{A}_s\sim\mathcal{B}^\lambda_{s}\times \mathcal{E}^\lambda_{s},$$
where $\mathcal{B}^\lambda/S_\lambda$ and $\mathcal{E}^\lambda/S_\lambda$ are locally complete families of abelian varieties of dimension $l$ and $g-l$ respectively. The index $\lambda$ encodes the structure of the isogeny and we let $\Lambda_l$ (or $\Lambda_l(\mathcal{A})$ if there is a risk of confusion about the family $\mathcal{A}/S$) be the set of all such $\lambda$. Given a positive integer $l'< l$, we will also be concerned with loci $S_{\lambda,\mu}\subset S_{\lambda}$ along which 
$$\mathcal{B}_s^\lambda\sim \mathcal{D}_s^{\lambda,\mu}\times\mathcal{F}_s^{\lambda,\mu},$$
where $\mathcal{D}^{\lambda,\mu}/S_{\lambda,\mu}$ and $\mathcal{F}^{\lambda,\mu}/S_{\lambda,\mu}$ are locally complete families of abelian varieties of dimension $l'$ and $l-l'$ respectively. Again, the index $\mu$ encodes the structure of the isogeny and we let $\Lambda_{l'}^{\lambda}$ be the set of all such $\mu$.\\

For our applications, we will mostly consider the case $(l,l')=(g-1,2)$. Given a positive integer $k$ and upon passing to a generically finite cover of $S_\lambda$  and $S_{\lambda,\mu}$, we have projections
\begin{align*}
p_{\lambda}: &\mathcal{A}_{S_{\lambda}}^k\to ({\mathcal{B}^{\lambda}})^k,\\
p_{\mu}: &(\mathcal{B}^{\lambda}_{S_{\lambda,\mu}})^k\to (\mathcal{D}^{\lambda,\mu})^k,
\end{align*}
and we let
$$p_{\lambda,\mu}:=p_{\mu}\circ p_{\lambda}: \mathcal{A}_{S_{\lambda,\mu}}^k\to (\mathcal{D}^{\lambda,\mu})^k,  \qquad  \text{for }\mu\in \Lambda_{l'}^{\lambda}.$$
For $\lambda\in \Lambda_l$, $\mu\in \Lambda_{l'}^\lambda$, and abelian varieties $B,D,F$ in the families $\mathcal{B}^\lambda/S_\lambda$, $\mathcal{D}^{\lambda,\mu}/S_{\lambda,\mu},$ and $\mathcal{F}^{\lambda,\mu}/S_{\lambda,\mu}$, we can consider the loci
\begin{align*}S_{\lambda}(B)&=\{s\in S_{\lambda}: \mathcal{B}^{\lambda}_s\cong B\}\subset S_{\lambda},\\
S_{\lambda,\mu}(D,F)&=\{s\in S_{\lambda,\mu}: \mathcal{D}^{\lambda,\mu}_s\cong D,\mathcal{F}^{\lambda,\mu}_s\cong F\}\subset S_{\lambda,\mu},
\end{align*}
which have the same dimension as the moduli stack of abelian $(g-l)$-folds.\\

The main result of this section is Proposition \ref{prop1} which is a generalization of Proposition 2.4 from \cite{V}. It states the following: 
\begin{quote}\textit{Given a family of irreducible subvarieties $\mathcal{Z}\subset \mathcal{A}^k$ that satisfies certain technical assumptions, there exists a $\lambda\in \Lambda_l$ such that the variety $p_\lambda(\mathcal{Z}_s)\subset B^k$ varies with $s\in S_\lambda(B)$, for $B$ in the family $\mathcal{B}^\lambda$.}
\end{quote}
The technical assumptions mentioned above are non-degeneracy conditions. Given a subvariety $\mathcal{Z}\subset \mathcal{A}^k/S$, we consider the following subsets of $S$:\\
\begin{align*}R_{gf}&=\bigcup_{\lambda}\{s\in S_\lambda : p_{\lambda}|_{\mathcal{Z}_s}: \mathcal{Z}_s\to (\mathcal{B}^{\lambda}_s)^k\text{ is generically finite on its image}\},\\
R_{ab}&=\bigcup_{\lambda}\{s\in S_\lambda : p_{\lambda}(\mathcal{Z}_s)\text{ is not an abelian subvariety of }(\mathcal{B}^{\lambda}_s)^k\},\\
R_{st}&=\bigcup_{\lambda}\{s\in S_\lambda : p_{\lambda}(\mathcal{Z}_s) \text{ is not stabilized by a positive-dimensional abelian subvariety of } (\mathcal{B}^{\lambda}_s)^k\}.
\end{align*}

\noindent We define three non-degeneracy conditions on $\mathcal{Z}\subset \mathcal{A}^k$:
\begin{itemize}
\setlength\itemsep{0em}
\item[($*$)] $R_{gf}\subset S$ is dense,
\item[($**$)] $R_{ab}\cap R_{gf}\subset S$ is dense,
\item [($***$)] $R_{st}\cap R_{gf}\subset S$ is dense.
\end{itemize}
These sets and conditions depend on $\mathcal{Z}$ and $l$ and, while $\mathcal{Z}$ should usually be clear from the context, we will say $(*)$ holds for a specified value of $l$. Note that we have the following obvious implications:
$$(***)\implies (**)\implies (*).$$ Assuming condition $(**)$ is satisfied, a simple argument will allow us to restrict ourselves to the case where condition $(***)$ is satisfied. Moreover, we will show in Lemma \ref{*implies**} that if $\mathcal{Z}\subset \mathcal{A}^k$ is foliated by positive dimensional suborbits then $(*)\iff (**)$. It follows that $(*)$ is the only crucial assumption on $\mathcal{Z}$ for our applications.\\

Recall that, given varieties $X$ and $S$, an irreducible subvariety $\mathcal{Z}\subset X_S$  with generic fiber of dimension $d$ over $S$ gives rise to a morphism from (an open in) $S$ to the Chow variety $\text{Chow}_d(X)$. Here $\text{Chow}_d(X)$ parametrizes effective cycles of class $[\mathcal{Z}_s]$ on $X$. We say that the variety $\mathcal{Z}_s$ varies with $s\in S$ if the corresponding morphism from $S$ to $\text{Chow}_d(X)$ is (generically) finite. We remind the reader of the notational convention of Remark \href{convention}{\ref{convention}} which allows us to remove the words in parenthesis from the previous sentences. The precise formulation of our main proposition is:

\begin{proposition}\label{prop1}
Let $\mathcal{Z}\subset \mathcal{A}^k$ be a variety that dominates $S$, has irreducible fibers of dimension $d$ over $S$, and satisfies condition $(**)$. Then, there exists a $\lambda\in \Lambda_l$ such that the subvariety
$$p_\lambda(\mathcal{A}_{S_{\lambda}(B)})\subset \big(\mathcal{B}^\lambda_{S_\lambda(B)}\big)^k\cong B^k_{S_{\lambda}(B)}$$
gives rise to a finite morphism from $S_{\lambda}(B)$ to $\textup{Chow}_d(B^k)$ for all $B$ in the family $\mathcal{B}^\lambda/S_\lambda$.
\end{proposition}

We propose to use this proposition to prove Conjecture \href{Vweakconj}{\ref{Vweakconj}} as outlined in the introduction: Consider a locally complete family of abelian $(2k-1)$-folds $\mathcal{A}_1/S_1$ and suppose that the set 
$$\{s\in S_1: \mathcal{A}_{1,s} \text{ contains a }k\text{-gonal curve}\}\subset S_1$$
is not contained in a countable union of Zariski closed subsets of $S_1$. Up to passing to a generically finite cover over $S_1$, we get a subvariety $\mathcal{Z}_1\subset\mathcal{A}_{1}^k$ whose fibers over $S_1$ are irreducible curves, and such that $\mathcal{Z}_{1,s}$ is contained in a fiber of $\mathcal{A}_{1,s}^k\to CH_0(\mathcal{A}_{1,s})$ for any $s\in S_1$. Moreover, using translation, we can arrange for $\mathcal{Z}_1$ to be in the kernel of the summation map $\mathcal{A}_1^{k}\to \mathcal{A}_1$. It is easy to show that $(**)$ holds for $l=2k-2$ in this setting. Hence, we can use the Proposition \ref{prop1} to get a family of surfaces foliated by suborbits:
$$\mathcal{Z}_2:=\varphi_{\mathcal{B}^\lambda}\left(p_{\lambda}(\mathcal{Z}_{1,S_\lambda})\right)\subset \varphi_{\mathcal{B}^\lambda}\left((\mathcal{B}^\lambda)^{k,\, 0}\right)=\varphi_{\mathcal{B}^\lambda}(\mathcal{B}^\lambda)^{k,\, 0}.$$
Indeed, observe that $p_{\lambda}(\mathcal{Z}_{1,s})$ varies with $s\in S_\lambda(B)$ if and only if the fiber of $\varphi_{\mathcal{B}^{\lambda}}(p_\lambda(\mathcal{Z}_{1,S_\lambda}))$ over the point $\varphi_{\mathcal{B}^\lambda}(S_\lambda(B))\in \varphi_{\mathcal{B}^\lambda}(S_\lambda)$ has dimension greater than the relative dimension of $p_\lambda(\mathcal{Z}_{1,S_\lambda})$ over $S_\lambda$. Indeed, the fiber of $\varphi_{\mathcal{B}^{\lambda}}(p_\lambda(\mathcal{Z}_{1,S_\lambda}))$ over this point is 
$$\bigcup_{s\in S_\lambda(B)}p_\lambda(\mathcal{Z}_{1,s})\subset B^k.$$

Rename the base $\varphi_{\mathcal{B}^\lambda}(S_\lambda)$ and the locally complete family $\varphi_{\mathcal{B}^\lambda}(\mathcal{B}^\lambda)/\varphi_{\mathcal{B}^\lambda}(S_\lambda)$ as $S_2$ and $\mathcal{A}_2/S_2$. We can hope to apply Proposition \href{prop1}{\ref{prop1}} to $\mathcal{Z}_2$ and proceed by induction. At the $i^{\text{th}}$ step, one gets a variety $\mathcal{Z}_{i}\subset \mathcal{A}_i^{k,\, 0}$, where $\mathcal{A}_i/S_i$ is a locally complete family of abelian $(2k-i)$-folds. The variety $\mathcal{Z}_i$ has relative dimension $i$, irreducible fibers over $S_i$, and is foliated by positive dimensional suborbits. At the $(2k-2)^{\text{nd}}$ step, one obtains a variety $\mathcal{Z}_{2k-2}\subset \mathcal{A}_{2k-2}^{k,\, 0}$ of relative dimension $2k-2$ which is foliated by positive-dimensional suborbits. Since $\dim_{S_{2k-2}}\mathcal{A}_{2k-2}=2$, Corollary \href{absurfbound}{\ref{absurfbound}} then provides the desired contradiction. The key issue with this induction argument is to ensure that condition $(**)$ is satisfied at each step. Since the relative dimension of the variety $\mathcal{Z}_i$ to which we apply Proposition \href{prop1}{\ref{prop1}} grows, it gets harder and harder to ensure generic finiteness of the projection $p_{\lambda}|_{\mathcal{Z}_{i,S_{i,\lambda}}}$. In the next section, we will present a way around this using the fact that $\mathcal{Z}_i$ is obtained by successive specializations and projections.\\

Before proceeding to give a summary of the proof of Proposition \ref{prop1}, we will establish some basic facts related to conditions $(*)$ and $(**)$. The main point is that if $\mathcal{A}/S$ is an almost locally complete family of abelian varieties and $\mathcal{Z}\subset \mathcal{A}^k$ is foliated by positive dimensional suborbits, then conditions $(*)$ and $(**)$ are equivalent for $l\geq 2$.\\

Given an abelian variety $A$, we denote by $T_A:=T_{A,0_A}$ the tangent space to $A$ at $0_A\in A$.
We let 
\begin{align*}\mathscr{T}&:=T_{\mathcal{A}|S},\\
%T^{\lambda}/S_{\lambda}&:=T_{{\mathcal{A}'}^\lambda}/S_\lambda\\
%T^{\lambda,\mu}/S_{\lambda,\mu}&:=T_{\mathcal{B}^{\lambda,\mu}}/S_{\lambda,\mu}\\
G/S&:=\text{Gr}(g-l,\T)/S,\\
G'/S&:=\text{Gr}(g-l',\T)/S,
%G^{\lambda}/S_{\lambda}&:=\text{Gr}(g-l-l',T^{\lambda})/S_{\lambda}
\end{align*}
and we consider the following sections
\begin{align*}&\sigma_{\lambda}: S_\lambda\to G_{{S_\lambda}}=\text{Gr}(g-l,\T_{S_\lambda}),\;\qquad \qquad \qquad \sigma_{\lambda}(s):=T_{\ker(p_{\lambda,s})},\\
%\sigma_{\mu}&: S_\mu\to G^{\lambda}_{{S_{\lambda,\mu}}}=\text{Gr}(g-l-l',T^{\lambda}_{S_{\lambda,\mu}}),\qquad \sigma_{\mu}(s):=T_{\ker(p_{\mu,s}),0}\\
&\sigma_{\lambda,\mu}: S_{\lambda,\mu}\to G'_{{S_{\lambda,\mu}}}= \text{Gr}(g-l',\T_{S_{\lambda,\mu}}),\qquad \;\;\sigma_{\lambda,\mu}(s):=T_{\ker(p_{\lambda,\mu,s})}. \end{align*}
Finally, we will denote by $\mathcal{H}/G$ and $\mathcal{H}'/G'$ the universal families of hyperplanes and by $\mathcal{H}^{\lambda}/S_\lambda$ and ${\mathcal{H}'}^{\lambda,\mu}/S_{\lambda,\mu}$ the pullbacks of the universal families of hyperplanes by $\sigma_\lambda$ and $\sigma_{\lambda,\mu}$. 

\begin{lemma}\label{dense}
For any $0\leq l'<l<g$, the following subsets are dense: 
\begin{align*}&\bigcup_{\lambda\in \Lambda_l}\sigma_\lambda(S_{\lambda})\subset G\\
\bigcup_{\lambda\in \Lambda_l}&\bigcup_{\mu\in \Lambda^\lambda_{l'}}\sigma_{\lambda,\mu}(S_{\lambda,\mu})\subset G'.\end{align*}
\end{lemma}
\begin{proof}
We only show this for the first subset. The same argument can be used for the second. It suffices to consider the locus of abelian varieties isogenous to $E^g$ for some elliptic curve $E$. This locus is dense in $S$ and given $s$ in this locus and any $M\in M_{k\times (g-l)}(\mathbb{Z})$ of rank $(g-l)$, the image of the tangent space $T_{E^{g-l}_{M}}$ under the differential of the isogeny is contained in $\sigma_\lambda(S_{\lambda})$ for some $\lambda\in \Lambda_l$. To finish the proof, it suffices to observe that the following subset is dense:
$$\{T_{E^{g-l}_{M}}: M\in M_{k\times (g-l)}(\mathbb{Z}),\; \text{rank} (M)=g-l\}\subset \text{Gr}(g-l,T_{E^g}).$$
\end{proof}

For the rest of this section, $\mathcal{A}/S$ is a locally complete family of abelian varieties of dimension at least $2$ and $\mathcal{Z}\subset \mathcal{A}^k$ is a subvariety which dominates $S$ and has irreducible fibers.

\begin{lemma}
If $\mathcal{Z}\subset \mathcal{A}^k$ is foliated by positive dimensional suborbits, then for very general $s\in S$ the subset $\mathcal{Z}_s$ is not an abelian subvariety of the form $A^r_M$ with $r>0$.
\end{lemma}
\begin{proof}
Consider the Zariski closed sets
$$\{s\in S: \mathcal{Z}_s=(\mathcal{A}_{s})^r_M\}\subset S.$$
There are countably many such sets so it suffices to show that none of them is all of $S$. Suppose that $\mathcal{A}^r_M$ is foliated by positive dimensional suborbits. By the previous lemma, there is a $\lambda\in \Lambda_2$ such that $p_\lambda((\mathcal{A}_{s})^r_M)=(\mathcal{B}^\lambda_s)^r_M$ is also foliated by positive dimensional suborbits for generic $s\in S_\lambda$. This contradicts Corollary \href{notfoliated2}{\ref{notfoliated2}}.
\end{proof}

\begin{lemma}\label{*implies**}
If $\mathcal{Z}\subset \mathcal{A}^k/S$ is foliated by positive dimensional orbits and $l\geq 2$, the conditions $(*)$ and $(**)$ are equivalent.
\end{lemma}
\begin{proof}
Condition $(**)$ obviously implies condition $(*)$, so it suffices to show the other implication. First, observe that if $p_{\lambda}|_{\mathcal{Z}_s}: \mathcal{Z}_s\to (\mathcal{B}_s^\lambda)^k$ is generically finite on its image, then its image is foliated by positive dimensional suborbits. In this case, $R_{gf}\cap S_\lambda$ is open in $S_\lambda$ and, for very general $s\in R_{gf}\cap S_\lambda$, the abelian variety $\mathcal{B}_s^\lambda$ is simple. By Lemma \href{notfoliated2}{\ref{notfoliated2}}, for such an $s$ the closed subset $p_\lambda(\mathcal{Z}_s)$ cannot be an abelian subvariety of $(\mathcal{B}^\lambda_s)^k$. Hence, $R_{gf}\cap R_{ab}\cap S_\lambda$ is dense in $S_\lambda$. This shows that $R_{ab}\cap R_{gf}\subset S$ is dense provided that $R_{gf}$ is.
\end{proof}

\subsection{Outline of the proof of Proposition \ref{prop1}}
\vspace{1em}
\noindent Let us now sketch the proof of Proposition \ref{prop1} by highlighting the key steps. 
\begin{itemize}
\item \emph{Step (I)}: We reduce to the case where $(***)$ is satisfied. We will show that if $(**)$ is satisfied but $(***)$ is not, there is an $r\in \mathbb{Z}_{>0}$ and an $M\in M_{k\times r}(\mathbb{Z})$ such that $\mathcal{Z}/\mathcal{A}^r_M$ satisfies $(**)$ and $(***)$. We can then proceed with the argument replacing $\mathcal{Z}\subset\mathcal{A}^k$ with $\mathcal{Z}/\mathcal{A}^r_M\subset \mathcal{A}^k/\mathcal{A}^r_M$.
\item \emph{Step (II)}: As in the proof of Proposition 2.4 from \cite{V}, we argue that, though the projections $p_\lambda|_{\mathcal{Z}_{S_\lambda}}$ are only defined along loci $S_\lambda$, there is a way to realize many of them birationally as specializations of a map defined over all of $G$. Namely, we show that there is a variety $\mathcal{Z}'_G/G$ and a generically finite morphism $p: \mathcal{Z}_G\to \mathcal{Z}'_G$ that identifies birationally to 
$$p_\lambda|_{\mathcal{Z}_{S_\lambda}}: (\mathcal{Z}_{G})_{\sigma_\lambda(S_\lambda)}=\mathcal{Z}_{S_\lambda}\to p_\lambda(\mathcal{Z}_{S_\lambda})$$
along many loci $\sigma_\lambda(S_\lambda)\subset G$. This uses crucially the density result from Lemma \ref{dense}.
\item \emph{Step (III)}: Resolving the singularities of $\mathcal{Z}_G$ and $\mathcal{Z}'_G$, we obtain a morphism $\widetilde{p}: \widetilde{\mathcal{Z}}_G\to \widetilde{\mathcal{Z}'}_G$ which coincides with $p$ up to birational equivalence. We examine the map
$$\widetilde{p}_*\circ \widetilde{j}^*: \text{Pic}^0(\mathcal{A}_t^k)\xrightarrow{\widetilde{j}^*} \text{Pic}^0(\mathcal{Z}_t)\xrightarrow{\widetilde{p}_*} \text{Pic}^0(\mathcal{Z}'_t),$$
where $\widetilde{j}$ is the composition of the natural map $\rho: \widetilde{\mathcal{Z}_G}\to \mathcal{Z}_G$ and the inclusion $\mathcal{Z}_G\to \mathcal{A}_G^k$. This is a non-zero morphism of abelian varieties and, applying the convention of Remark \ref{convention}, we see that there is a $r\in \{0,\ldots, k-1\}$ and a matrix $M\in M_{k\times r}(\mathbb{Z})$ such that $\ker(\widetilde{p}_*\circ \widetilde{j}^*)$ contains $\text{Pic}^0(\mathcal{A}_t)^r_M$ as a finite index subgroup for every $t\in G$.\\

\indent But if $\widetilde{p}$ agrees birationally with the map $p_\lambda\circ \rho: (\widetilde{\mathcal{Z}_G})_{\sigma_\lambda(S_\lambda)}\to p_\lambda(\mathcal{Z}_{S_\lambda})$ along $\sigma_\lambda(S_\lambda)$ and $p_\lambda(\mathcal{Z}_{s})$ does not vary with $s\in S_\lambda$, the abelian variety
$$\text{Pic}^0\left(\widetilde{p_\lambda(\mathcal{Z}_{s})}\right)$$
admits an isogeny from the abelian variety $\text{Pic}^0\left((\mathcal{E}^\lambda_s)^k/(\mathcal{E}^\lambda_s)^r_M\right)$
for any $s\in S_\lambda(B)$. Since the set of abelian varieties in the locally complete family $\varphi_{\mathcal{E}^\lambda}(\mathcal{E}^\lambda)/\varphi_{\mathcal{E}^\lambda}(S_\lambda)$ admitting a non-zero morphism to a fixed abelian variety is discrete, we reach the desired contradiction.
\end{itemize}
%We will see in the course of the proof of Proposition 3.2 that if $R_{gf}$ is non-empty then it must be dense. Moreover note that $p_{\lambda}^{(2)}|_{\mathcal{Z}_s}$ being generically finite on its image is an open condition on $S_\lambda$.

%The families $\mathcal{B}^\lambda/S_\lambda$ and $\mathcal{D}^{\lambda,\mu}\times_{S_{\lambda,\mu}}\mathcal{F}^{\lambda,\mu}/S_{\lambda,\mu}$ induce maps from $S_\lambda$ and $S_{\lambda,\mu}$ to moduli spaces of abelian varieties with some fixed polarizations and the sets above are the fibers of these maps.\\

\subsection{Step (I): Reduction to the validity of condition $(***)$}

\vspace{0.5em}
\begin{lemma}\label{(***)}
If $\mathcal{Z}$ satisfies condition $(**)$ for some $l$, there is an $r\in \{1,\ldots, k\}$ and an $M\in M_{k\times r}(\mathbb{Z})$ such that $\mathcal{Z}/\mathcal{A}^r_M\subset \mathcal{A}^k/\mathcal{A}^{r}_M$ satisfies conditions $(**)$ and $(***)$ for $l$.
\end{lemma}

\begin{proof}
Since $(**)$ is satisfied, we can find a $\lambda_0\in \Lambda_l$ and an $s_0\in S_{\lambda_0}$, such that $p_{\lambda_0}|_{\mathcal{Z}_{s_0}}$ is generically finite on its image and $\mathcal{B}_{s_0}^{\lambda_0}$ is simple. Suppose that $p_{\lambda_0}(\mathcal{Z}_{s_0})$ is stabilized by $(\mathcal{B}_{s_0}^{\lambda_0})_M^r$ but not by a larger abelian subvariety of $(\mathcal{B}_{s_0}^{\lambda_0})^k$. Then, the subvariety$$p_{\lambda_0}(\mathcal{Z}_{s_0})/(\mathcal{B}_{s_0}^{\lambda_0})_M^r\subset(\mathcal{B}_{s_0}^{\lambda_0})^k/ \mathcal(\mathcal{B}_{s_0}^{\lambda_0})_M^r$$
is not stabilized by any abelian subvariety of $(\mathcal{B}_{s_0}^{\lambda_0})^k/ \mathcal(\mathcal{B}_{s_0}^{\lambda_0})_M^r$. We claim that $\mathcal{Z}/\mathcal{A}^r_M\subset \mathcal{A}^k/\mathcal{A}^r_M$ satisfies $(**)$ and $(***)$.\\

The idea is to use the Gauss map to leverage the fact that $p_{\lambda_0}(\mathcal{Z}_{s_0})/(\mathcal{B}_{s_0}^{\lambda_0})_M^r\subset(\mathcal{B}_{s_0}^{\lambda_0})^k/ \mathcal(\mathcal{B}_{s_0}^{\lambda_0})_M^r$ is not stabilized by an abelian subvariety in order to deduce similar information about $p_\lambda(\mathcal{Z}_{S_\lambda})/(\mathcal{A}_{S_\lambda})^r_M$ for $\lambda\neq \lambda_0$. For each $\lambda\in \Lambda_l$ such that $p_{\lambda}|_{\mathcal{Z}_{S_\lambda}}$ is generically finite on its image, we have a commutative diagram:\\
\vspace{1em}
 \begin{equation}\label{lambdadiag1}
 \begin{tikzcd}
 \mathcal{Z}_{S_{\lambda}} \ar[r,dashed,"g"] \arrow{d}[swap]{p_{\lambda}}& \text{Gr}(d,\T_{S_{\lambda}}^k) \ar[d,dashed,"\pi_{\mathcal{H}^{\lambda}}"]\\
 p_{\lambda}(\mathcal{Z}_{S_{\lambda}}) \ar[r,dashed,"g"]& \text{Gr}\big(d,(\T_{S_{\lambda}}/\mathcal{H}^{\lambda})^k\big),
 \end{tikzcd}
 \end{equation}
where $g$ denotes the Gauss map and $\pi_{\mathcal{H}^{\lambda}}$ is the rational map induced by the quotient map $\T_{S_{\lambda}}^k\to (\T_{S_{\lambda}}/\mathcal{H}^{\lambda})^k$. Note that in this diagram and in what follows we take the liberty to write $p_\lambda$ for the map $p_\lambda|_{\mathcal{Z}_{{S}_\lambda}}: \mathcal{Z}_{S_\lambda}\to p_\lambda(\mathcal{Z}_{S_\lambda})$. We can consider the analogous diagram obtained by quotienting by $\mathcal{A}^{r}_M$:
 \begin{equation}
 \begin{tikzcd}\label{lambdadiag2}
 \mathcal{Z}_{S_{\lambda}}/(\mathcal{A}_{S_\lambda})^r_M \ar[r,dashed,"g"] \arrow{d}[swap]{p_{\lambda}}& \text{Gr}\Big(d,\T_{S_{\lambda}}^k/(\T_{S_{\lambda}})^r_M\Big) \ar[d,dashed,"\pi_{\mathcal{H}^{\lambda}}"]\\
 p_{\lambda}(\mathcal{Z}_{S_{\lambda}})/(\mathcal{A}_{S_\lambda})^r_M \ar[r,dashed,"g"]& \text{Gr}\Big(d,\T_{S_{\lambda}}^k/[(\mathcal{H}^{\lambda})^k+(\T_{S_{\lambda}})^r_M]\Big).
 \end{tikzcd}
 \end{equation}
Here we abuse notation by writing $p_{\lambda}$ and $\pi_{\mathcal{H}^{\lambda}}$ for the maps induced by $p_{\lambda}$ and $\pi_{\mathcal{H}^{\lambda}}$ on the quotients. Note that it is not a priori clear that the bottom Gauss map takes value in the Grassmanian of $d$-planes. This is ensured by the fact that $p_{\lambda}: \mathcal{Z}_{S_\lambda}\to p_{\lambda}(\mathcal{Z}_{S_\lambda})$ is generically finite on its image. We show this in the following lemma:

\begin{lemma}Consider an abelian variety $A\sim B\times E$, where $B$ and $E$ are abelian varieties of smaller dimension. Let $p: A^k\to {B}^k$ be the composition of the projection with the isogeny. Given $r\in \{1,\ldots, k\}$ and $M\in M_{k\times r}(\mathbb{Z})$, denote also by $p$ the map $A^k/A_{M}^r\to {B}^k/{B}_{M}^r$ induced by $p$. If $Z\subset A^k$ is such that $p|_{Z}: Z\to B^k$ is generically finite on its image, then $p|_{Z/A_M^r}$ is also generically finite on its image.
\end{lemma}
\begin{proof}
Since $p|_{Z/A_M^r}$ is proper and locally of finite presentation, it suffice to show that it is quasi-finite over an open in its image. The fiber of $p|_{Z/A_M^r}: Z/A^r_M\to {B}^k/{B}^r_M$ over $p(z)\in p(Z)/{B}^r_M$ is the set of all $A^r_M$-cosets of the fiber of $p$ over $p(z)$. Hence, the fiber of $p|_{Z/A_M^r}$ over $p(z)$ is finite whenever the fiber of $p|_{Z}$ over $p(z)$ is finite.
\end{proof}

Now, the key observation is that the upper right corner of diagram \ref{lambdadiag2} is a specialization along $\sigma_\lambda(S_\lambda)\subset G$ of the following diagram defined over $G:=\text{Gr}(g-l,\T)$:

 \begin{equation}\label{Gdiag}
 \begin{tikzcd}
 \mathcal{Z}_G/(\mathcal{A}_G)^r_M \ar[r,dashed,"g"] & \text{Gr}\Big(d,\T_{G}^k/(\T_{G})^r_M\Big) \ar[d,dashed,"\pi_{\mathcal{H}}"]\\
&  \text{Gr}\Big(d,\mathscr{T}_G^k/[\mathcal{H}^k+(\mathscr{T}_G)^r_M]\Big).
 \end{tikzcd}
 \end{equation}
 
 Moreover, the composition of these rational maps is well-defined. Indeed, it is defined over $\sigma_{\lambda_0}(s_0)\in G$ since $p_{\lambda_0}|_{\mathcal{Z}_{s_0}/(\mathcal{A}_{s_0})^r_M}$ is generically finite on its image and diagram \ref{lambdadiag2} is commutative. In fact, the composition $\pi_{\mathcal{H}}\circ g$ is generically finite on its image over $\sigma_{\lambda_0}(s_0)\in G$
%$$\pi_{\mathcal{H}^{\lambda}}\circ g: \mathcal{Z}_G/\mathcal{A}^r_M\dashrightarrow \text{Gr}(d,\mathscr{T}_G^k/[\mathcal{H}^k+(\mathscr{T}_G)^r_M])$$
% is well defined and generically finite on its image over $\sigma_{\lambda_0}(s_0)$. Indeed, 
% $$\pi_{\mathcal{H}^{\lambda_0}}\circ g=g\circ p_{\lambda_0}$$
since both $p_{\lambda_0}|_{\mathcal{Z}_{s_0}/(\mathcal{A}_{s_0})^r_M}$ and the Gauss map of $p_{\lambda_0}(\mathcal{Z}_{s_0}/(\mathcal{A}_{s_0})^r_M)$ are generically finite on their images. For the Gauss map, this follows from results of Griffiths and Harris (see (4.14) in \cite{GH}) along with the fact that $p_{\lambda_0}(\mathcal{Z}_{s_0}/(\mathcal{A}_{s_0})^r_M)\subset (\mathcal{B}^{\lambda_0}_{s_0})^k/(\mathcal{B}^{\lambda_0}_{s_0})^r_M$ is not stabilized by an abelian subvariety.\\
 
 Hence, there is an open $U\subset G$ over which $\pi_{\mathcal{H}}\circ g$ is a well-defined rational map and generically finite on its image. We claim that for any $\lambda\in \Lambda_l$ the set $\sigma_\lambda(S_\lambda)\cap U$ maps to $R_{st}\cap R_{gf}$ under the natural map from $G$ to $S$. This implies that condition $(***)$ is satisfied since the following subset is dense:
 $$\bigcup_{\lambda\in \Lambda_l}\sigma_\lambda(S_\lambda)\cap U.$$
 Our claim follows from the fact that $g\circ \pi_{\mathcal{H}^\lambda}$ is generically finite on its image over any $s$ such that $\sigma_\lambda(s)\in U$. Because diagram \ref{lambdadiag2} is commutative, this implies both that $p_{\lambda}(\mathcal{Z}_{s}/(\mathcal{A}_{s})^r_M)\subset (\mathcal{B}^{\lambda}_{s})^k/(\mathcal{B}^{\lambda}_{s})^r_M $ is not stabilized by an abelian subvariety and that $p_{\lambda}: \mathcal{Z}_{s}/(\mathcal{A}_{s})^r_M\to p_\lambda(\mathcal{Z}_{s}/(\mathcal{A}_{s})^r_M)$ is generically finite.
\end{proof}

Lemma \ref{(***)} allows us to reduce to the case where $\mathcal{Z}\subset \mathcal{A}^k$ satisfies condition $(***)$ to prove Proposition \ref{prop1}. Indeed, if $\lambda\in \Lambda_l$ and $B$ in the family $\mathcal{B}^\lambda$ are such that the family 
$$p_{\lambda}(\mathcal{Z}_{S_{\lambda}(B)}/\mathcal{A}^r_{M})\subset ({B}^k/{B}^r_{M})_{S_\lambda(B)}$$
gives rise to a finite morphism from $S_{\lambda}(B)$ to $\text{Chow}_d({B}^k/{B}^{r}_M)$, then the family 
$$p_{\lambda}(\mathcal{Z}_{S_{\lambda}(B)})\subset {B}^k_{S_\lambda(B)}$$
also gives rise to a finite morphism from $S_{\lambda}(B)$ to $\text{Chow}_d(B^k)$. Hence, if the conclusion of Proposition \ref{prop1} holds for $\mathcal{Z}/\mathcal{A}^r_M$, then it also holds for $\mathcal{Z}$ itself. While everything we will do from now on is valid with $\mathcal{Z}/\mathcal{A}^{r}_M$ satisfying $(***)$, we will keep writing $\mathcal{Z}\subset \mathcal{A}^k$ in an effort to simplify the notation.\\

\subsection{Step (II): Birational factorization}

In the previous section we saw that, given $\mathcal{Z}\subset \mathcal{A}^k$ satisfying condition $(***)$, we have a rational map
$$q:=\pi_{\mathcal{H}}\circ g: \mathcal{Z}_G\dashrightarrow \text{Gr}\big(d,(\mathscr{T}_G/\mathcal{H})^k\big),$$
which is generically finite on its image, and, along each locus $\sigma_\lambda(S_\lambda)$, a factorization
\begin{equation}
\begin{tikzcd}\label{diagfact}
\mathcal{Z}_{\sigma_\lambda(S_\lambda)}\cong \mathcal{Z}_{S_\lambda} \ar[rr,dashed, "q:=\pi_\mathcal{H}\circ g"] \ar[dr,swap,"p_\lambda"] &\; & \text{Gr}\big(d, (\mathscr{T}_G/\mathcal{H})^k\big)\\
 &  p_{\lambda}(\mathcal{Z}_{S_{\lambda}}).\ar[ur,swap,dashed,"g"]& 
\end{tikzcd}
\end{equation}

In this section our goal is to show that there exists a factorization of $q$ over $G$ that identifies birationally with $p_\lambda$ along $\sigma_\lambda(S_\lambda)$ for many $\lambda\in \Lambda_l$. This is a key input in Step (III). While the content of this section is likely known to experts, we decided to spell it out at length because it plays a decisive role in the argument.\\

\begin{lemma}\label{cov1}
Consider $\mathcal{Z}/S$, a family with irreducible fibers and base, and $q: \mathcal{Z}/S\to \mathcal{X}/S$ such that $q_s: \mathcal{Z}_s\to \mathcal{X}_s$ is generically finite for each $s\in S$. Let $S'\subset S$ be a Zariski dense subset such that for each $s'\in S'$ we have a factorization of $q$ over $s'$ as follows:
$$
\begin{tikzcd}\label{diagfact}
\mathcal{Z}_{s'} \ar[rr, "q"] \ar[dr,swap,"f_{s'}"] &\; & \mathcal{X}_s\\
 &  f(\mathcal{Z}_{s'}).\ar[ur,swap,"g_{s'}"]& 
\end{tikzcd}
$$
Then there is a family $\mathcal{Z}'/S$, morphisms $p: \mathcal{Z}\to \mathcal{Z}'$ and $p': \mathcal{Z}'\to \mathcal{X}$, and a Zariski dense subset $S''\subset S'$ such that:
\begin{itemize}
\setlength\itemsep{0em}
\item For any $s''\in S''$, the varieties $p(\mathcal{Z}_{s''})$ and $f(\mathcal{Z}_{s''})$ are birational,
\item  $p$ and $p'$ induce the same morphism on function fields as $f_{s''}$ and $g_{s''}$ respectively over any $s''\in S''$.
\end{itemize}
\end{lemma}
\begin{proof}
We first restrict to a Zariski open subset of $\mathcal{X}$ (which we call $\mathcal{X}$ in keeping with Remark \ref{convention}) over which $q$ is finite \'etale and such that $\mathcal{X}\to S$ is smooth. By work of Hironaka, we can find a compactification $\overline{\mathcal{X}}$ of $\mathcal{X}$ with simple normal crossing divisors at infinity. Restricting to an open in the base, we can assume that $\overline{\mathcal{X}}_s\setminus \mathcal{X}_s$ is an snc divisor for any $s\in S$. One can use a version of Ehresmann's lemma allowing for an snc divisor at infinity to see that $\mathcal{X}\to S$ is a locally-trivial fibration in the category of smooth manifolds.\\

It follows that we get a covering
\[q: \mathcal{Z}/S\to \mathcal{X}/S.\] Note that we renamed as $\mathcal{Z}$ an open subset of $\mathcal{Z}$ over which $q$ is \'etale. To complete the proof of Lemma \ref{cov1} we will need the following Lemma.

\begin{lemma}\label{cov2}
Consider a covering $q: \mathcal{Z}/S\to \mathcal{X}/S$, with $\mathcal{Z}_s$ connected for every $s\in S$, and a factorization of $q$ over $s_0\in S$:
\[\begin{tikzcd}
\mathcal{Z}_{s_0}\ar[rr,"q"] \ar{dr}[swap]{f_{s_0}} & &\mathcal{X}_{s_0} \\
&f_{s_0}(\mathcal{Z}_{s_0}). \ar{ur}[swap]{g_{s_0}} & 
\end{tikzcd}\]
Then, there is a factorization
\[\begin{tikzcd}
\mathcal{Z}\ar[rr,"q"] \ar{dr}[swap]{f} & &\mathcal{X} \\
&f(\mathcal{Z}), \ar{ur}[swap]{g} & 
\end{tikzcd}\]
which identifies with the original factorization over $s_0\in S$.
\end{lemma}
\begin{proof}
Consider the Galois closure $\mathcal{Z}'\to \mathcal{X}$ of the covering $q: \mathcal{Z}\to \mathcal{X}$. Note that $\mathcal{Z}'_{s_0}$ is connected. Indeed, there is a normal subgroup of the Galois group of $\mathcal{Z}'/\mathcal{X}$ corresponding to deck transformations inducing the trivial permutation of the connected components of $\mathcal{Z}'_{s_0}$. This subgroup corresponds to a Galois cover of $\mathcal{Z}$ since $\mathcal{Z}_{s_0}$ is connected.\\

It follows from the fact that $\mathcal{Z}'_{s_0}$ is connected that the map $\text{Gal}(\mathcal{Z}'/\mathcal{X})\to \text{Gal}(\mathcal{Z}'_{s_0}/\mathcal{X}_{s_0})$ is injective because a deck transformation which is the identity on the base and on fibers must be the identity. Since $\text{Gal}(\mathcal{Z}'/\mathcal{X})$ has order 
$$d:=\deg(\mathcal{Z}'/\mathcal{X})=\deg(\mathcal{Z}_{s_0}'/\mathcal{X}_{s_0}),$$
and $\text{Gal}(\mathcal{Z}'_{s_0}/\mathcal{X}_{s_0})$ has order at most $d$, this restriction morphism must be an isomorphism and $\mathcal{Z}_{s_0}'/\mathcal{X}_{s_0}$ is thus also Galois. One then has an equivalence of categories between the poset of intermediate coverings of $\mathcal{Z}'/\mathcal{X}$ and that of $\mathcal{Z}_{s_0}'/\mathcal{X}_{s_0}$, and hence between the poset of intermediate coverings of $\mathcal{Z}/\mathcal{X}$ and that of $\mathcal{Z}_{s_0}/\mathcal{X}_{s_0}$.
\end{proof}

By the previous lemma, to each factorization $f_{s'}$ we can associate an intermediate cover $\mathcal{Z}\to \mathcal{Z}^{s'}$ of $\mathcal{Z}\to \mathcal{X}$ which agrees with $f_{s'}$ at $s'$. Since there are only finitely many intermediate covers of $\mathcal{Z}\to \mathcal{X}$, we get a partition of $S'$ according to the isomorphism type of the cover $\mathcal{Z}\to \mathcal{Z}^{s'}$. One subset $S''\subset S'$ of this partition must be dense in $S$. Let $f: \mathcal{Z}\to f(\mathcal{Z})$ be the corresponding intermediate cover.
\end{proof}

For the rest of the proof of Proposition \href{prop1}{\ref{prop1}} we are back in the situation of diagram (\ref{diagfact}).

\begin{corollary}
There is a variety $\mathcal{Z}'/G$, a morphism $p: \mathcal{Z}\to \mathcal{Z}'$, and a subset $\Lambda_{l,0}\subset \Lambda_{l}$, such that:
\begin{itemize}
\setlength\itemsep{0em}
\item $\bigcup_{\lambda\in \Lambda_{l,0}}\sigma_{\lambda}(S_{\lambda})\subset G$ is dense,
\item $p_{\lambda}(\mathcal{Z}_t)$ and $p(\mathcal{Z}_t)$ are birational for any $\lambda\in \Lambda_{l,0}$, $t\in \sigma_{\lambda}(S_\lambda)$,
\item $p: \mathcal{Z}_t\to p(\mathcal{Z}_t)$ and $p_{\lambda}: \mathcal{Z}_t\to p_\lambda(\mathcal{Z}_t)$ induce the same map on function fields, for any $\lambda\in \Lambda_{l,0}$, $t\in \sigma_{\lambda}(S_\lambda)$.
\end{itemize}
\end{corollary}
\begin{proof}
This follows from the previous lemma and its proof once we observe that the intermediate covering of $q$ (or rather of an appropriate \'etale restriction of $q$ as above) associated to the factorization $p_{\lambda}: \mathcal{Z}_t\to p_{\lambda}(\mathcal{Z}_t)$ is independent of $t\in \sigma_{\lambda}(S_\lambda)$.
\end{proof}

%\begin{remark}
%A technical point is that the maps $p_\lambda$ are a priori only defined after passing to a generically finite cover of $S_\lambda$. This does not cause problems as the covering $p_{\lambda}: \mathcal{Z}_{s}\to p_{\lambda}(\mathcal{Z}_s)$ is defined without passing to a generically finite cover and it does not change under base change by a generically finite map $S_\lambda'\to S_\lambda$ containing $s$ in its image. 
%\end{remark}

\vspace{1em}
\subsection{Step (III): Final argument}
\begin{proof}[Proof of Proposition \ref{prop1}]
Using Step (I), we may assume that $\mathcal{Z}\subset \mathcal{A}^k$ satisfies condition $(***)$. By Step (II), we have a variety $\mathcal{Z}'_G/G$, a morphism $p: \mathcal{Z}_G\to \mathcal{Z}'_G$, and a subset $\Lambda_{l,0}\subset \Lambda_l$, such that: \begin{itemize} 
\setlength\itemsep{0em}
\item $p$ identifies birationally with $p_{\lambda}: \mathcal{Z}_{S_\lambda}\to p_\lambda(\mathcal{Z}_{S_\lambda})$ along $\sigma_\lambda(S_\lambda)\subset G$, for all $\lambda\in \Lambda_{l,0}$,
\item $\bigcup_{\lambda\in \Lambda_{l,0}}\sigma_\lambda(S_\lambda)\subset G$ is dense.
\end{itemize}
Up to shrinking $G$, we can consider desingularizations with smooth fibers:
$$\widetilde{p}: \widetilde{\mathcal{Z}_G}/G\to \widetilde{\mathcal{Z}_G'}/G.$$
We have the inclusion
$$j: \mathcal{Z}_G\to \mathcal{A}_G^k$$
as well as the map
$$\widetilde{j}:=j\circ \rho: \widetilde{\mathcal{Z}_G}\to \mathcal{A}_G^k,$$
where $\rho: \widetilde{\mathcal{Z}_G}\to \mathcal{Z}_G$ is the natural map.
The morphism $\widetilde{j}$ gives rise to a pullback map
$$\widetilde{j}^*: \text{Pic}^0(\mathcal{A}_t)\to \text{Pic}^0(\mathcal{\widetilde{Z}}_t).$$
Since $\widetilde{p}$ is generically finite on its image we can consider the composition
$$\widetilde{p}_*\circ \widetilde{j}^* :  \text{Pic}^0(\mathcal{A}^k_t)\to \text{Pic}^0(\widetilde{\mathcal{Z}}_t)\to \text{Pic}^0(\widetilde{\mathcal{Z}_t'}).$$
This is a morphism of abelian varieties defined for every $t\in G$.\\
\;\\

We first show that it is non-zero along $\sigma_{\lambda}(S_\lambda)$ for any $\lambda\in \Lambda_{l,0}$. Consider $t\in \sigma_{\lambda}(S_\lambda)$. The variety $\mathcal{A}_t$ is isogenous to $\mathcal{B}_t^{\lambda}\times \mathcal{E}_t^{\lambda}$ and we have the following commutative diagram
$$
\begin{tikzcd}
\mathcal{Z}_t \arrow{r}{j} \arrow{d}[swap]{{p_\lambda}} & \mathcal{A}_t^k \arrow{d}{p_{\lambda}}\\
p_{\lambda}(\mathcal{Z}_t) \arrow{r}{j'}& (\mathcal{B}_t^{\lambda})^k .
\end{tikzcd}
$$

Consider a desingularization of $p_\lambda(\mathcal{Z}_t)$ and the induced map
$$\widetilde{j}': \widetilde{p_\lambda(\mathcal{Z}_t)}\to  (\mathcal{B}_t^{\lambda})^k.$$
The fact that $\widetilde{p}$ identifies birationally to $p_{\lambda}$ along $\sigma_\lambda(S_\lambda)$ implies that the following diagram is commutative:
$$
\begin{tikzcd}
\text{Pic}^0(\widetilde{\mathcal{Z}}_t) & \text{Pic}^0(\mathcal{A}_t^k)  \arrow{l}[swap]{\widetilde{j}^*}\\
\text{Pic}^0(\widetilde{\mathcal{Z}}'_t)\cong \text{Pic}^0\Big(\widetilde{p_\lambda(\mathcal{Z}_t)}\Big) \arrow{u}{\widetilde{p}^*}  & \text{Pic}^0\left((\mathcal{B}_t^{\lambda})^k\right) \arrow{l}[swap]{\widetilde{j}'^*} \arrow{u}[swap]{p_{\lambda}^*} .
\end{tikzcd}
$$
It follows that
\begin{align*}\widetilde{p}_*\circ (\widetilde{j}^*\circ {p_{\lambda}}^*)=\widetilde{p}_*\circ (\widetilde{p}^*\circ \widetilde{j}'^*)=(\widetilde{p}_*\circ \widetilde{p}^*)\circ \widetilde{j}'^*=[\deg(\widetilde{p})]\circ \widetilde{j}'^*.\end{align*}
Since $p_\lambda(\mathcal{\mathcal{Z}}_t)$ is positive dimensional, the morphism $\widetilde{j}'^*$ is non-zero and so $\widetilde{p}_*\circ \widetilde{j}^*$ is non-zero.\\
\;\\

We conclude that the connected component of the identity of the kernel of $\widetilde{p}_*\circ \widetilde{j}^*$ is an abelian subscheme of $\mathcal{A}^k$ which is not all of $\mathcal{A}^k$. For very general $t\in G$, the abelian variety $\mathcal{A}_t$ is simple. Therefore, for such a $t$, the abelian subvariety 
$$\ker(\widetilde{p}_*\circ \widetilde{j}^*)^0_t:=\ker\left(\widetilde{p}_*\circ \widetilde{j}^*:   \text{Pic}^0(\mathcal{A}^k_t)\to \text{Pic}^0(\widetilde{\mathcal{Z}_t'})\right)^0$$ is of the form $(\mathcal{A}_{t})_M^r$, with $M\in M_{k\times r}(\mathbb{Z})$ of rank $r$, and $r\leq k-1$. Choosing $M$ and $r$ such that 
$$\big\{t\in G: \ker(\widetilde{p}_*\circ \widetilde{j}^*)_t^0=(\mathcal{A}_{t})_M^r\big\}\subset G$$
is dense, and observing that this set is locally closed, we see that, shrinking $G$ if needed, $\ker(\widetilde{p}_*\circ \widetilde{j}^*)_t^\circ=(\mathcal{A}_{t})_M^r$ for all $t\in G$. In particular, for $t\in \sigma_{\lambda}(S_\lambda)$, the abelian variety
$$\text{Pic}^0\big(\ker(p_{\lambda})_t\big)/\ker (\widetilde{p}_*\circ \widetilde{j}^*)_t\cap\text{Pic}^0\big(\ker(p_{\lambda})_t\big)$$
is isogenous to the abelian variety
$$\text{Pic}^0\big(\ker({p_{\lambda}})_t\big)/(\mathcal{A}_{t})_M^r\cap\text{Pic}^0\big(\ker(p_{\lambda})_t\big)\neq 0.\\$$
\;\\

Now consider $\lambda\in \Lambda_l$ such that $\sigma_\lambda(S_\lambda)\neq \emptyset$, namely such $\sigma_{\lambda}(S_\lambda)$ has survived the various base change by generically finite maps, and $B\in \mathcal{B}^\lambda$ such that $\sigma_\lambda(S_{\lambda}(B))\neq \emptyset$. Suppose that there is a curve 
$$C\subset  \sigma_\lambda(S_{\lambda}(B))\cong S_\lambda(B),$$
such that $p_\lambda(\widetilde{Z}_t)=p_\lambda(\widetilde{Z}_{t'})$ for any $t,t'\in C$, namely such that $C$ is contracted by the morphism from $S_{\lambda}(B)$ to $\text{Chow}_d(B^k)$ associated to the family $p_{\lambda}(\mathcal{Z}_{S_{\lambda}(B)})\subset B^k_{S_\lambda(B)}$. Since the abelian variety 
$$\text{Pic}^0(\widetilde{\mathcal{Z}}_t')\cong \text{Pic}^0\Big(\widetilde{p_\lambda(\mathcal{Z}_t)}\Big)$$ does not depend on $t\in C$, it must admit an isogeny from each of the abelian varieties
$$\text{Pic}^0\big(\ker({p_{\lambda}})_t\big)/(\mathcal{A}_{t})_M^r\cap\text{Pic}^0\big(\ker(p_{\lambda})_t\big),\qquad t\in C.$$ Because an abelian variety cannot admit an isogeny from every member of a non-isotrivial family of abelian varieties,
 this provides the desired contradiction. This completes the proof of Proposition \href{prop1}{\ref{prop1}}.
\end{proof}

\section{Salvaging generic finiteness and a proof of Voisin's conjecture}\label{salvaging}
\vspace{0.5em}
In this section, we refine the results from the previous section in order to bypass assumption $(*)$ in the inductive application of Proposition \href{prop1}{\ref{prop1}}. The idea is quite simple: In the last section we saw that we can specialize to abelian varieties $\mathcal{A}_s$ isogenous to a product $B\times E_s$, where $E_s$ is an elliptic curve, in such a way that the image of $\mathcal{Z}_s\subset \mathcal{A}_s^k$ under the projection $\mathcal{A}_s^k\to B^k$ varies with $s$. In this section, we will specialize to abelian varieties $\mathcal{A}_s$ isogenous to a product $D\times F\times E_s$, where $E_s$ is an elliptic curve and $D$ an abelian variety of dimension at least $2$.\\

Under suitable assumptions, the image of $\mathcal{Z}_s$ under the projections $\mathcal{A}_s^k\to D^k$ and $\mathcal{A}_s^k\to (D\times F)^k$ varies with $s$. Hence, if we consider in $(D\times F)^k$ and $D^k$ the union of the image of these projections for every $s$, we get varieties $Z_1\subset (D\times F)^k$ and $Z_2\subset D^k$ of dimension $\dim_S \mathcal{Z}+1$. Since the restriction of the projection $(D\times F)^k\to D^k$ to $Z_1$ has image $Z_2$ and $\dim Z_1=\dim Z_2$, this restriction is generically finite on its image. This allows us to get condition $(*)$ for free when applying Proposition \href{prop1}{\ref{prop1}} inductively. We spend this section making this simple idea rigorous and deducing a proof of Theorem \href{2k-2}{\ref{2k-2}}.\\

\subsection{Salvaging generic finiteness} The main result of this section is the following.
\begin{proposition}\label{mainprop}
Suppose that $\mathcal{Z}\subset \mathcal{A}^k$ satisfies $(**)$ for $l'\geq 2$. There exists a $\lambda\in \Lambda_{(g-1)}$ such that 
$$\varphi_{\mathcal{B}^\lambda}\big(p_\lambda(\mathcal{Z}_{S_\lambda})\big)\subset \varphi_{\mathcal{B}^\lambda}(\mathcal{B}^\lambda)^k$$ satisfies $(*)$ for $l'$ and has relative dimension $\dim_S \mathcal{Z}+1$ over $\varphi_{\mathcal{B^\lambda}}(S_\lambda)$.
\end{proposition}

%The families $\mathcal{B}^\lambda/S_\lambda$, $\lambda\in \Lambda_{(g-1)}$ introduced in the last section are families of $(g-1)$-dimensional abelian varieties with some polarization type $\theta^\lambda$. Hence, they give rise to a diagram
%\begin{equation}
%\begin{tikzcd}
%(\mathcal{B}^\lambda)^k \arrow{r}{\varphi_\lambda} \arrow{d}& \mathcal{A'}^k \arrow{d}\\
%S_\lambda \arrow{r}{\psi_\lambda} & S',
%\end{tikzcd}
%\end{equation}
%where $\mathcal{A'}/S'$ is the universal family over the moduli stack of abelian varieties of dimension $(g-1)$ with polarization $\theta^\lambda$. Note that $\mathcal{A'}/S'$ depends of course on $\lambda$ but we suppress $\lambda$ from the notation. We will think of $\mathcal{A}'/S'$ as another locally complete family of abelian varieties. This family comes with its own set $\Lambda_{l'}'$ indexing loci $S_{\eta'}$ along which $\mathcal{A}'_s\sim \mathcal{B'}_s^\eta\times\mathcal{E'}^\eta_s$. Let $\varphi_{\lambda,\mu}$ be the composition of $\varphi_\lambda|_{(\mathcal{B}^\lambda_{S_{\lambda,\mu}})^k}$ with 
%$$$(\mathcal{D}^{\lambda,\mu})^k\to (\mathcal{D}^{\lambda,\mu}\times \mathcal{F}^{\lambda,\mu})^k\to (\mathcal{B}^{\lambda}_{S_{\lambda,\mu}})^k,$$
%where the last map is the isogeny encoded by $\mu$. We get a diagram
%\begin{equation}
%\begin{tikzcd}
%(\mathcal{D}^{\lambda,\mu})^k \arrow{r}{\varphi_{\lambda,\mu}} \arrow{d}& (\mathcal{B'}^\eta)^k \arrow{d}\\
%S_{\lambda,\mu} \arrow{r} & S'_{\eta},
%\end{tikzcd}
%\end{equation}
%where $\eta$ is some index in $\Lambda'_{l'}$, and $\psi_{\lambda}(S_{\lambda,\mu})=S_\eta'$.

The proof of this proposition will rest on the following lemma:

\begin{lemma}\label{salv}
Let $\mathcal{A}/S$ be a locally complete family of abelian $g$-folds and consider $\lambda\in \Lambda_{(g-1)}$ and $\mu\in \Lambda^{\lambda}_{l'}$, where $l'<g-1$. If $p_{\lambda,\mu}(\mathcal{Z}_t)\subset D^k$ varies with $t\in S_{\lambda,\mu}(D,F)$, then $\varphi_{\mathcal{B}^\lambda}\big(S_{\lambda,\mu}(D,F)\big)$ lies in $R_{gf}\subset \varphi_{\mathcal{B}^\lambda}(S_\lambda)$.\end{lemma}
Note that the curve $S_{\lambda,\mu}(D,F)$ is contracted by the map $\varphi_{\mathcal{B}^\lambda}$ from $S_\lambda$ to the moduli stack of polarized abelian $(g-1)$-folds with an appropriate polarization type. 
\begin{proof}Given $\lambda\in \Lambda_{(g-1)}(\mathcal{A})$, the family $\varphi_{\mathcal{B}^\lambda}\big(\mathcal{B}^\lambda\big)$ is locally complete. For any $\mu\in \Lambda_{l'}^\lambda(\mathcal{A})$, there is a $\mu'\in \Lambda_{l'}\big(\varphi_{\mathcal{B}^\lambda}\big(\mathcal{B}^\lambda\big)\big)$ making the following diagram commute:
$$
\begin{tikzcd}
\mathcal{A}_{S_{\lambda, \mu}}^k \ar[rr, "\varphi_{\mathcal{D}^{\lambda,\mu}}\circ p_{\lambda, \mu}"] \ar[dr,swap,"\varphi_{\mathcal{B}^\lambda}\circ p_{\lambda}"] &\; & \varphi_{\mathcal{D}^{\lambda,\mu}}(\mathcal{D}^{\lambda,\mu})^k\\
 &  \varphi_{\mathcal{B}^\lambda}\big(\mathcal{B}^\lambda_{S_{\lambda,\mu}}\big)^k.\ar[ur,swap,"\varphi_{\mathcal{D}^{\lambda,\mu}}\circ p_{\mu'}"]& 
\end{tikzcd}
$$
Consider the restriction of this diagram to $S_{\lambda,\mu}(D,F)$, where $D$ and $F$ are members of the families $\mathcal{D}^{\lambda,\mu}$ and $\mathcal{F}^{\lambda,\mu}$:
\begin{equation}\label{comm}
\begin{tikzcd}
\mathcal{A}_{S_{\lambda, \mu}(D,F)}^k \ar[rr, "\varphi_{\mathcal{D}^{\lambda,\mu}}\circ p_{\lambda, \mu}"] \ar[dr,swap,"\varphi_{\mathcal{B}^\lambda}\circ p_{\lambda}"] &\; & D^k\\
 &  \varphi_{\mathcal{B}^\lambda}\Big(\mathcal{B}^\lambda_{S_{\lambda,\mu}(D,F)}\Big)^k.\ar[ur,swap,"\varphi_{\mathcal{D}^{\lambda,\mu}}\circ p_{\mu'}"]& 
\end{tikzcd}
\end{equation}
%Since the curve $S_{\lambda,\mu}(D,F)$ is contracted by the map $\varphi_{\mathcal{B}^\lambda}$, we have an isomorphism
%$$\varphi_{\mathcal{B}^\lambda}\big(\mathcal{B}^\lambda_{S_{\lambda,\mu}(D,F)}\big)\cong \mathcal{B}^\lambda_s,$$
%for any $s\in S_{\lambda,\mu}(D,F)$. Moreover, the map
%$$\varphi_{\mathcal{D}^{\lambda,\mu}}\circ p_{\mu'}: \varphi_{\mathcal{B}^\lambda}\Big(\mathcal{B}^\lambda_{S_{\lambda,\mu}(D,F)}\Big)^k\to D^k$$
%coincides with $p_\mu: (\mathcal{B}_s^\lambda)^k\to D^k$.\\

%By hypothesis on $\lambda$ and up to shrinking $S_\lambda$, the subvariety $p_\lambda(\mathcal{Z}_t)\subset \mathcal{B}^\lambda_{S_{\lambda,\mu}(D,F)}$ varies with $t\in S_{\lambda,\mu}(D,F)$ for any $D,F$. It follows that
%$$\dim (\varphi_{\mathcal{B}^\lambda}\circ p_{\lambda})\left(\mathcal{Z}_{S_{\lambda,\mu}(D,F)}\right)=\dim_S\mathcal{Z}+1.$$
If $p_{\lambda,\mu}(\mathcal{Z}_t)$ varies with $t\in S_{\lambda,\mu}(D,F)$, we have
$$\dim (\varphi_{\mathcal{D}^{\lambda,\mu}}\circ p_{\lambda,\mu})\left(\mathcal{Z}_{S_{\lambda,\mu}(D,F)}\right)=\dim_S\mathcal{Z}+1$$
and thus the restriction of $\varphi_{\mathcal{D}^{\lambda,\mu}}\circ p_{\lambda, \mu}$ to the $(\dim_S\mathcal{Z}+1)$-dimensional variety $\mathcal{Z}_{S_{\lambda,\mu}(D,F)}\subset \mathcal{A}^k_{S_{\lambda,\mu}(D,F)}$ is generically finite on its image.
From the commutativity of diagram \ref{comm}, we see that the restriction of $\varphi_{\mathcal{D}^{\lambda,\mu}}\circ p_{\mu'}$ to 
$$(\varphi_{\mathcal{B}^\lambda}\circ p_{\lambda})\left(\mathcal{Z}_{S_{\lambda,\mu}(D,F)}\right)\subset  \varphi_{\mathcal{B}^\lambda}\Big(\mathcal{B}^\lambda_{S_{\lambda,\mu}(D,F)}\Big)^k$$ is generically finite on its image. It follows that $\varphi_{\mathcal{B}^\lambda}(S_{\lambda,\mu}(D,F))\in R_{gf}$.
\end{proof}
\begin{remark}\label{rem}
Note that if $p_{\lambda,\mu}(\mathcal{Z}_t)$ varies with $t\in S_{\lambda,\mu}(D,F)$, then $p_{\lambda}(\mathcal{Z})$ varies with $t\in S_{\lambda,\mu}(D,F)$. In particular, it follows that $(\varphi_{\mathcal{B}^\lambda}\circ p_\lambda) (\mathcal{Z}_{S_\lambda})$ has relative dimension $\dim_S\mathcal{Z}+1$ over $\varphi_{\mathcal{B^\lambda}}(S_\lambda)$.\end{remark}

\begin{proof}[Proof of Proposition \ref{mainprop}]
In light of Remark \ref{rem} and Lemma \ref{salv}, it suffices to show that there is a $\lambda_0\in \Lambda_l$ and a subset $\Lambda_{l',0}^\lambda\subset \Lambda^{\lambda}_{l'}$ such that:
\begin{itemize}
\item $\bigcup_{\mu\in \Lambda_{l',0}^\lambda}\sigma_{\lambda,\mu}(S_{\lambda,\mu})\subset G'_{S_\lambda} \text{ is dense,}$
\item $p_{\lambda,\mu}(\mathcal{Z}_t)$ varies with $t\in S_{\lambda,\mu}(D,F)$ for every $D,F$.
\end{itemize}
Following the same argument as in the proof of Lemma \ref{cov1}, we get a partition
$$\bigcup_{\lambda\in \Lambda_{(g-1)}}\Lambda_{l'}^{\lambda}=T_1\sqcup T_2\sqcup\ldots \sqcup T_n$$ according to the isomorphism type of the \'etale covering associated to the map $p_{\lambda,\mu}$. Consider the following function:
\begin{align*}&D:\Lambda_{(g-1)}\longrightarrow\{I\subset \{1,\ldots, n\}: I\neq\emptyset\}\\
\lambda\mapsto \Big\{i\in &\{1,\ldots, n\}: \bigcup_{\mu\in T_i\cap \Lambda_{l'}^\lambda}\sigma_{\lambda,\mu}(S_{\lambda,\mu})\subset G'_{S_{\lambda,\mu}}\text{ is dense }\Big\}.\end{align*}
Its fibers make up a partition of $\Lambda_{(g-1)}$ and so there is a fiber $D^{-1}(I)$ such that
$$\bigcup_{\lambda\in D^{-1}(I)}\sigma_{\lambda}(S_\lambda)\subset G\text{ is dense}.$$
Pick $i_0\in I$ and let $T=T_{i_0}$. By construction, the following subset is dense for any $\lambda\in D^{-1}(I)$:
$$\bigcup_{\mu\in T\cap \Lambda_{l'}^\lambda}\sigma_{\lambda,\mu}(S_{\lambda,\mu})\subset G'_{S_{\lambda,\mu}}.$$
It follows that the following subset is also dense:
$$\bigcup_{\lambda\in D^{-1}(I)}\bigcup_{\mu\in T\cap \Lambda_{l'}^\lambda}\sigma_{\lambda,\mu}(S_{\lambda,\mu})\subset G'.$$
We can carry out the proof of Proposition \ref{prop1} for $l'$ using the loci $S_{\lambda,\mu}$ with $\lambda\in D^{-1}(I)$ and $\mu\in T$ instead of the loci $S_{\eta}$ with $\eta\in\Lambda_{l'}$. It follows that $p_{\lambda,\mu}(\mathcal{Z}_t)$ varies with $t\in S_{\lambda,\mu}(D,F)$ for every $D$ and $F$ if $\lambda\in D^{-1}(I)$ and $\mu\in T$. Indeed, the key ingredient of Proposition \ref{prop1} is that $\mathcal{Z}$ satisfies $(**)$ along with the density statement of Lemma \ref{dense}. Choosing $\lambda\in D^{-1}(I)$ and setting $\Lambda_{l',0}^{\lambda}:=T\cap \Lambda_{l'}^{\lambda}$ completes the proof of this proposition.
\end{proof}

\vspace{1em}
\subsection{Main results}
\vspace{1em}

\begin{corollary}
Suppose that a very general abelian variety of dimension $g$ has a positive dimensional orbit of degree $k$ and let $A$ be a very general abelian variety of dimension $(g-l)\geq 2$. Then, $A^k$ contains an $(l+1)$-dimensional subvariety foliated by positive dimensional suborbits.
\end{corollary}
\begin{proof}
Under the assumption of this corollary, we have a positive dimensional suborbit $\mathcal{Z}\subset \mathcal{A}^k/S$, where $\mathcal{A}\to S$ is a locally complete family of abelian $g$-folds. We can apply Proposition \ref{mainprop} inductively since condition $(**)$ follows from $(*)$ by Lemma 3.3.
\end{proof}

\begin{corollary}\label{conjpf}
Conjecture 1.3 holds: a very general abelian variety of dimension $\geq 2k-1$ has no positive dimensional orbits of degree $k$.
\end{corollary}
\begin{proof}
Note that a suborbit $\mathcal{Z}\subset \mathcal{A}^k$ of relative dimension $d$ satisfies $(*)$ for $l\geq d$ since the following subset is dense:
$$\bigcup_{\lambda\in \Lambda_l}\sigma_\lambda(S_\lambda)\subset G.$$
Indeed, if $V$ is a vector space and $W\subset V^k$ has dimension $d<\dim V$, then the restriction of the projection $V^k\to (V/H)^k$ to $W$ is an isomorphism onto its image for generic $H\in \text{Gr}(\dim V-d,V)$. In particular, if $\mathcal{Z}\subset \mathcal{A}^k$ has relative dimension $1$ then it satisfies $(*)$ for any $1\leq l\leq g-1$. Hence, if a very general abelian $(2k-1)$-fold has a positive dimensional orbit of degree $k$, for a very general abelian surface $B$, the variety $B^{k,\, 0}$ is foliated by positive dimensional suborbits. By Corollary \href{absurfbound}{\ref{absurfbound}}, this does not hold for any abelian surface, let alone generically.
\end{proof}

\begin{theorem}\label{2k-2}
For $k\geq 3$, a very general abelian variety of dimension at least $2k-2$ has no positive dimensional orbits of degree $k$, i.e.
$\mathscr{G}(k)\leq 2k-2$.\end{theorem}
\begin{proof}
Let $\mathcal{A}/S$ be a locally complete family of abelian $2k-2$-folds and $\mathcal{Z}\subset \mathcal{A}^{k,\, 0}$ a one-dimensional normalized suborbit. By the previous corollary, there is a $\lambda\in \Lambda_2$ such that $(\varphi_{\mathcal{B}^\lambda}\circ p_\lambda)(\mathcal{Z}_{S_\lambda})/\varphi_{\mathcal{B}^\lambda}(S_\lambda)$ has relative dimension $2k-3$. This was obtained by successive degenerations and projections. But the morphism
\begin{align*}S_\lambda&\to\text{Chow}_d\Big((\mathcal{B}^{\lambda})^{k,\, 0}\Big)\\
s&\mapsto \left[(\varphi_{\mathcal{B}^{\lambda}}\circ p_\lambda)(\mathcal{Z}_s)\right]\end{align*}
is a $\binom{2k-3}{2}$-parameter family of suborbits in $(\varphi_{\mathcal{B}^\lambda}\circ p_\lambda)(\mathcal{Z}_{S_\lambda})$. Hence, $(\varphi_\lambda\circ p_\lambda)(\mathcal{Z}_{S_\lambda})$ must be foliated by suborbits of dimension at least $2$. This contradicts Corollary \href{absurfbound}{\ref{absurfbound}}.
\end{proof}

\begin{corollary}\label{gonalitybound}
For $k\geq 3$, a very general abelian variety of dimension at least $2k-2$ has gonality at least $k+1$. In particular Conjecture \href{Vweakconj}{\ref{Vweakconj}} holds.\end{corollary}

We can use the same argument to rule out the existence of two-dimensional orbits of small degree.

\begin{theorem}\label{no2dim}
A very general abelian variety of dimension at least $2k-4$ does not have a $2$-dimensional orbit of degree $k$ for $k\geq 4$.
\end{theorem}
\begin{proof}
Suppose $\mathcal{A}/S$ is a locally complete family of $(2k-4)$-dimensional abelian varieties and that $\mathcal{Z}\subset \mathcal{A}^k$ is a $2$-dimensional suborbit. Using the same argument as in the proof of Corollary \href{conjpf}{\ref{conjpf}}, we see that $\mathcal{Z}$ satisfies $(*)$, and thus $(**)$. We can therefore follow the proof of Theorem \href{2k-2}{\ref{2k-2}}.
\end{proof}

As mentioned in the introduction, our method yields stronger results if we only consider orbits of the form $|\sum_{i=1}^{k-l}\{a_i\}+l\{0_A\}|$.

\begin{theorem}\label{k+1} A very general abelian variety $A$ of dimension at least $2k+2-l$ does not have a positive dimensional orbit of the form $|\sum_{i=1}^{k-l}\{a_i\}+l\{0_A\}|$, i.e. $\mathscr{G}_{l}(k)\leq 2k+2-l.$ \\
Moreover, if $A$ is a very general abelian variety of dimension at least $k+1$, the orbit $|k\{0_A\}|$ is countable, i.e. $\mathscr{G}_{k}(k)\leq k+1.$
\end{theorem}
\begin{proof}
By the results of \cite{V}, it suffices to show that a very general abelian variety of dimension $2k+2-l$ has no positive dimensional orbits of the form $|\sum_{i=1}^{k-l}\{a_i\}+l\{0_A\}|$. If this were not the case, we could find $\mathcal{A}/S$, a locally complete family of $(2k+2-l)$-dimensional abelian varieties, and $\mathcal{Z}\subset \mathcal{A}^k$, a one-dimensional suborbit such that
$$\mathcal{A}_s^{k-l}\times \{0_{\mathcal{A}_s}\}^l\cap \mathcal{Z}_s\neq \emptyset,\qquad \forall s\in S.$$
By Proposition \href{mainprop}{\ref{mainprop}}, there is a $\lambda\in \Lambda_2$ such that $(\varphi_{\mathcal{B}^\lambda}\circ p_{\lambda})(\mathcal{Z}_{S_\lambda})$ has relative dimension $2k+1-l$. Given $B$ in the family $\mathcal{B}^\lambda$ and $\underline{b}=(b_1,\ldots, b_{k-l},0_{B},\ldots, 0_{B})
\in B^{k}$, consider
$$S_{\lambda}(B,\underline{b}):=\{s\in S_\lambda(B): \underline{b}\in (\varphi_{\mathcal{B}^\lambda}\circ p_{\lambda})(\mathcal{Z}_{s})\}.$$
Clearly, $(\varphi_{\mathcal{B}^\lambda}\circ p_{\lambda})(\mathcal{Z}_{S_\lambda(B,\underline{b})})$ is a suborbit. In particular, $(\varphi_{\mathcal{B}^\lambda}\circ p_{\lambda})(\mathcal{Z}_{S_\lambda(B)})$ is foliated by suborbits of codimension at most $2(k-l)$. This contradicts Corollary \href{absurfbound}{\ref{absurfbound}}. A similar argument shows that $\mathscr{G}_k(k)\leq k+1$.
\end{proof}

\vspace{1em}
\section{Further observations \& results}\label{further}
\vspace{0.5em}
We have seen how the minimal degree of a positive dimensional orbit gives a lower bound on the gonality of a smooth projective variety and used this to provide a new lower bound on the gonality of very general abelian varieties. In this section, we show how one can use results about the maximal dimension of orbits for rational equivalence of small degree in order to give lower bounds on other measures of irrationality for very general abelian varieties. We finish by discussing another conjecture of Voisin from \cite{V} and its implication for the gonality of very general abelian varieties.

\subsection{Measures of irrationality} Recall the definitions of some of the measures of irrationality for irreducible $n$-dimensional projective varieties:
\begin{align*}
\text{irr}(X)&:=\min \left\{\delta \mid \exists \text{ a degree } \delta\text{ rational covering } X\dashrightarrow \mathbb{P}^n\right\},\\
\text{cov.gon}(X)&:= \min \begin{cases}\delta\;\bigg|\begin{rcases} \text{for a generic point }
x\in X \text{, there is a curve}\\ C\subset X
\text{ such that } x\in C \text{ and } \text{gon}(C)=\delta\end{rcases}.\end{cases}
\end{align*}
We will also consider the following measure of irrationality which interpolates between the \textit{degree of irrationality} $\text{irr}(X)=\text{irr}_n(X)$ and the \textit{covering gonality} $\text{gon}(X)=\text{irr}_1(X)$:
$$\text{irr}_d(X):=\min \begin{cases}\delta\;\bigg|\begin{rcases}\text{for a general point }
x\in X \text{, there is a }d\text{-dimensional}\\\text{subvariety } Z\subset X
\text{ such that } x\in Z \text{ and } \text{irr}(Z)=\delta\end{rcases}.\end{cases}$$

Applying the methods of the previous section, we obtain the following:

\begin{corollary}
If $A$ is a very general abelian variety of dimension at least $2k-4$ and $k\geq 4$, then
$$\text{\textup{irr}}_2(A)\geq k+1.$$
\end{corollary}
\begin{proof}
A surface with degree of irrationality $k$ in a smooth projective variety $X$ provides a $2$-dimensional suborbit of degree $k$. The result then follows from Theorem \href{no2dim}{\ref{no2dim}}.
\end{proof}
Similarly, we can use upper bounds on the dimension of suborbits of degree $k$ to obtain lower bounds on the degree of irrationality of abelian varieties. To our knowledge, the best bound currently in the literature is 
$$\text{irr}(A)\geq \dim A+1,$$
for any abelian variety $A$. This result is due to Alzati and Pirola and is a consequence of their study of the behavior of the holomorphic length under dominant rational maps (see \cite{AP2}). It is an interesting fact that this bound also follows easily from Voisin's Theorem \href{k-1}{\ref{k-1}}. Indeed, a dominant rational map from $A$ to $\mathbb{P}^{\dim A}$ of degree at most $\dim A$ provides a $(\dim A)$-dimensional suborbit of degree $\dim A$. Note that Yoshihara and Tokunaga-Yoshihara (\cite{Y},\cite{TY}) give examples of abelian surfaces $A$ with $\text{irr}(A)=3$, so that the Alzati-Pirola bound is sharp for $\dim A=2$.\\

Our results allow us to show that the Alzati-Pirola bound is not optimal, at least for very general abelian varieties.
\begin{theorem}\label{no(k-1)}
Orbits of degree $k$ on a very general abelian variety of dimension at least $k-1$ have dimension at most $k-2$, for $k\geq 4$.
\end{theorem}

\begin{corollary}\label{sommese}
If $A$ is a very general abelian variety of dimension $g\geq 3$, then
$$\text{\textup{irr}}(A)\geq g+2.$$
\end{corollary}
\begin{proof}[Proof of Theorem \href{no(k-1)}{\ref{no(k-1)}}]
Consider $\mathcal{A}/S$, a locally complete family of abelian $g$-folds, and $\mathcal{Z}\subset \mathcal{A}^{k,\, 0}/S$, a $(k-1)$-dimensional suborbit. We claim that $\mathcal{Z}$ satisfies $(*)$ for $l=g-1$ if $g\geq k-1$. This is obvious if $g>k-1$ and will be shown below for $g=k-1$. Assuming this, by Lemma \ref{*implies**} and Proposition \ref{prop1}, for appropriate $\lambda\in \Lambda_{(k-2)}$ and $B$ in the family $\mathcal{B}^\lambda$, the subvariety $(\varphi_{\mathcal{B}^\lambda}\circ p_{\lambda})(\mathcal{Z}_{S_\lambda(B)})\subset B^{k,\, 0}$ is $k$-dimensional and foliated by $(k-1)$-dimensional suborbits. This contradicts Corollary \href{1dimfam}{\ref{1dimfam}}.\\

We now assume that $g=k-1$. To show that $\mathcal{Z}$ satisfies $(*)$ for $l=k-2$, we will need the following easy lemma which we give without proof:
\begin{lemma}
Consider a $g$-dimensional vector space $V,$ a positive integer $r$, and a $g$-dimensional subspace $W\subset V^r$. If the restriction of the projection $\pi_L: V^r\to (V/L)^r$ to $W$ is not an isomorphism for any $L\in \mathbb{P}(V)$, then $W=V_M$ for some $M\in \C^r$.
\end{lemma}

Suppose $\mathcal{Z}$ fails to satisfy $(*)$ for $l=k-2$. Then, for any $s\in S$ and $z\in (\mathcal{Z}_s)_{\text{sm}}$, the tangent space $T_{\mathcal{Z}_s,z}$ must be of the form $(T_{\mathcal{A}_s})_M\subset T_{\mathcal{A}_s}^{k,\, 0}$ for some $M\in \C^{k,\, 0}:=\{(x_1,\ldots, x_n)\in \C^n: \sum_{i=1}^n x_i=0\}$. Thus, for each $s\in S$, we get a morphism $(\mathcal{Z}_s)_{\text{sm}}\to \mathbb{P}(\C^{k,\, 0})$. For very general $s\in S$, the abelian variety $\mathcal{A}_s$ is simple and so Lemma \ref{absub} implies that $\mathcal{Z}_s$ cannot be stabilized by an abelian subvariety of $\mathcal{A}_s^k$. Hence, the Gauss map of $\mathcal{Z}_s$ is generically finite on its image and so is the morphism $(\mathcal{Z}_s)_{\text{sm}}\to \mathbb{P}(\C^{k,\, 0})$. It follows that the image of this morphism must contain an open in $\mathbb{P}(\C^{k,\, 0})$ and so a point $M\in \mathbb{P}(\mathbb{R}^{k,\, 0})$. Since the subspace $(T_{\mathcal{A}_s})_M\subset T_{\mathcal{A}_s}^{k,\, 0}$ is not totally isotropic for $\omega_k$ for any non-zero $\omega\in H^0(\mathcal{A}_s,\Omega^2)$, this provides the desired contradiction.
\end{proof}

Corollary \ref{sommese} is improved in the preprint \cite{CMNP} written in collaboration with E. Colombo, J. C. Naranjo, and G. P. Pirola. For a discussion of the degree of irrationality of abelian surfaces see the preprints \cite{C} and \cite{Martin}. In particular, in \cite{Martin} Corollary \ref{sommese} is extended to very general $(1,d)$-polarized abelian surfaces for $d\nmid 6$.\\

\subsection{Existence results for positive dimensional orbits}
When seeking lower bounds on measures of irrationality one is led to rule out existence of large dimensional orbits of small degree. Since the study of orbits for rational equivalence is interesting in its own right, one could instead seek existence results for subvarieties of $A^k$ foliated by $d$-dimensional suborbits. Alzati and Pirola show in Examples 5.2 and 5.3 of \cite{AP} that any abelian surface $A$ admits a $2$-dimensional suborbit of degree $3$ and a threefold in $A^{3,\, 0}$ foliated by suborbits of positive dimension. In particular, using the argument from Remark \href{trick}{\ref{trick}}, we see that Corollary \href{normabsurfbound}{\ref{normabsurfbound}} is sharp for $d=0,1,2$.

\begin{example}\label{HSL}
In \cite{HSL} Lin shows that Corollary \href{normabsurfbound}{\ref{normabsurfbound}} is sharp for every $d$. 
\end{example}

The methods of \cite{HSL} can be used to show the following:
\begin{proposition} \label{curvecase}
Let $A$ be an abelian $g$-fold which is a quotient of the Jacobian of a smooth genus $g'$ curve $C$. For any $k\geq g'+d-1$, the variety $A^{k,\, 0}$ contains a $(g(k+1-g'-d)+d)$-dimensional subvariety foliated by $d$-dimensional suborbits.
\end{proposition}
\begin{proof}
To simplify notation we identify $C$ with its image in $J(C)$. We can assume that $0_A\in C$. Recall that the summation map $\text{Sym}^lC\to J(C)$ has fibers $\mathbb{P}^{l-g'}$ for all $l\geq g'$. Moreover, if $(c_1,\ldots, c_l)$ and $(c_1',\ldots, c_l')$ are such that $\sum c_i=\sum c_i'\in J(C)$, then the zero-cycles $\sum \{c_i\}$ and $\sum \{c_i\}'$ are equal in $CH_0(C)$, and thus in $CH_0(A)$.\\

In light of Remark \href{trick}{\ref{trick}}, it suffices to show that $A^{g'+d-1,\, 0}$ contains a $d$-dimensional suborbit. Consider the following map:
\begin{align*}\psi: C\times C^{g'+d-1}&\to A^{g'+d-1}\\
(c_0,(c_1,\ldots, c_{g'+d-1}))&\mapsto (c_1-c_0,\ldots, c_{g'+d-1}-c_0).\end{align*}

This morphism is generically finite on its image since the restriction of the summation map $A^2\to A$ to $C^2\subset A^2$ is generically finite.
The intersection of the image of $\psi$ with $A^{g'+d-1,\, 0}$ is a $d$-dimensional suborbit. Indeed, given 
$$(c_1-c_0,\ldots, c_{g'+d-1}-c_0)\in \text{Im}(\psi)\cap A^{g'+d-1,\, 0},$$\
we have
$$\sum_{i=1}^{g'+d-1} c_i=(g'+d-1)c_0,$$
so that
$$\sum_{i=1}^{g'+d-1} \{c_i\}=(g'+d-1)\{c_0\}\in CH_0(C).$$
It follows that
$$
\sum_{i=1}^{g'+d-1}\{c_i-c_0\}= (g'+d-1)\{0_A\}\in CH_0(A).$$\end{proof}

\begin{remark}
The previous proposition is by no means optimal. Indeed, it provides a one-dimensional suborbit of degree $3$ for a very general abelian $3$-fold while, as pointed out by a referee, such an abelian variety admits a one-dimensional family of one-dimensional suborbits. To see this, observe that a very general abelian $3$-fold is isogenous to the Jacobian of a quartic plane curve $C$. Projecting from a point $c\in C$ gives a degree $3$ rational map $\varphi_c: C\dashrightarrow \mathbb{P}^1$. Let $Z\subset C\times \text{Sym}^3 C$ be the image of the rational map $C\times \mathbb{P}^1\dashrightarrow C\times \text{Sym}^3C$ taking a generic point $(c,t)\in C\times \mathbb{P}^1$ to $(c,\varphi_c^{-1}(t))$. The image of $Z$ under the map
\begin{align*}C\times &\text{Sym}^3C\longrightarrow \text{Sym}^3A\\
(c,x+y+z)&\mapsto \{3x+c\}+\{3y+c\}+\{3z+c\}\end{align*}
is easily checked to be a surface foliated by positive dimensional suborbits and contained in a fiber of the albanese.\\
\end{remark}

\subsection{Further conjectures}
We believe that Theorem \href{2k-2}{\ref{2k-2}} can be improved. In fact, though Conjecture \href{Vweakconj}{\ref{Vweakconj}} is the main conjecture of \cite{V}, it is not the most ambitious. Voisin proposes to attack Conjecture \href{Vweakconj}{\ref{Vweakconj}} by studying the locus $Z_A$ of positive dimensional normalized orbits of degree $k$:
$$Z_A:=\bigg\{a_1\in A: \exists\; a_2,\ldots, a_{k-1}: \dim\Big|\{a_1\}+\ldots+\{a_{k-1}\}+\Big\{-\sum_{i=1}^k a_i\Big\}\Big|>0\bigg\}.$$
In particular, she suggests to deduce Conjecture \href{Vweakconj}{\ref{Vweakconj}} from the following:
\begin{conjecture}[Voisin, Conj. 6.2 in \cite{V}]\label{norm}
If $A$ is a very general abelian variety then
$$\dim Z_A\leq k-1.$$
\end{conjecture}
Voisin shows that this conjecture implies Conjecture \href{Vweakconj}{\ref{Vweakconj}} but it in fact implies the following stronger conjecture:
\begin{conjecture}\label{Vstrong} A very general abelian variety of dimension at least $k+1$ does not have a positive dimensional orbit of degree $k$, i.e. $\mathscr{G}(k)\leq k+1$.\label{Vstrongconj}
\end{conjecture}
To see how Conjecture \ref{Vstrong} follows from Conjecture \ref{norm}, suppose that $A$, a very general abelian variety of dimension $k$, has a normalized positive dimensional orbit $|\{a_1\}+\ldots+\{a_{k-1}\}|$. By Remark \ref{trick}, for any $a\in A$, the normalized orbit $|\{(k-1)a\}+\{a_1-a\}+\ldots+\{a_{k-1}-a\}|$ is positive dimensional. It follows that $Z_A=A$ and so $\dim Z_A=k>k-1$.\\

As mentioned above, the results of Pirola and Alzati-Pirola give $\mathscr{G}(2)\leq 3$ and $\mathscr{G}(3)\leq 4$. Our main theorem provides the bound $\mathscr{G}(4)\leq 6$. An interesting question is to determine if $\mathscr{G}(4)\leq 5$. This would provide additional evidence in favor of Conjecture \href{Vstrongconj}{\ref{Vstrongconj}}.

 \subsection*{Aknowledgements} This paper owes a lot to the work of Alzati, Pirola, and Voisin. I thank M. Nori and A. Beilinson for countless useful and insightful conversations as well as for their support. I would also like to extend my gratitude to C. Voisin for bringing my attention to this circle of ideas by writing \cite{V}, and for kindly answering some questions about the proof of Theorem 1.1 from that article. Finally, I am grateful for the advice of the referees which considerably improved the exposition.

\bibliography{biblio}

\end{document}